\documentclass[reqno,10pt,centertags]{amsart}
\usepackage{amsmath,amsthm,amscd,amssymb,latexsym,upref,stmaryrd}
\usepackage{hyperref}

\newcommand{\bbC}{{\mathbb{C}}}

\newcommand{\bbN}{{\mathbb{N}}}
\newcommand{\bbR}{{\mathbb{R}}}

\newcommand{\cB}{{\mathcal B}}

\newcommand{\cH}{{\mathcal H}}

\newcommand{\cU}{{\mathcal U}}

\newcommand{\cX}{{\mathcal X}}

\newcommand{\no}{\notag}
\newcommand{\lb}{\label}
\newcommand{\f}{\frac}

\newcommand{\ol}{\overline}

\newcommand{\Oh}{O}

\newcommand{\ran}{\text{\rm{ran}}}

\newcommand{\dom}{\text{\rm{dom}}}

\newcommand{\supp}{\text{\rm{supp}}}

\newcommand{\bi}{\bibitem}

\newcommand{\tr}{\text{\rm{tr}}}
\newcommand{\Sect}{\text{\rm Sect}}
\newcommand{\cl}{\text{\rm{cl}}}

\newcommand{\lbrac}{\left[\negthickspace}
\newcommand{\rbrac}{\negthickspace\right]}

\newcommand{\La}{\Lambda}

\newcommand{\ga}{\gamma}

\newcommand{\De}{\Delta}
\newcommand{\de}{\delta}
\newcommand{\te}{\theta}

\newcommand{\CR}{{\bbC\backslash\sigma(\Hte)}}

\newcommand{\gate}{\ensuremath{\ga_{\te_0,\te_R}}}
\newcommand{\gates}{\ensuremath{\ga_{\te_0^{\prime},\te_R^{\prime}}}}

\newcommand{\Lates}{\ensuremath{\La_{\te_0,\te_R}^{\te_0^{\prime},\te_R^{\prime}}}}

\newcommand{\Lazzqq}{\ensuremath{\La_{0,0}^{\frac{\pi}{2},\frac{\pi}{2}}}}

\newcommand{\Lades}{\ensuremath{\La_{\de_0,\de_R}^{\de_0^{\prime},\de_R^{\prime}}}}

\newcommand{\Ste}{\ensuremath{S_{\te_0,\te_R}}}

\newcommand{\Gte}{\ensuremath{G_{\te_0,\te_R}}}
\newcommand{\Hte}{\ensuremath{H_{\te_0,\te_R}}}

\renewcommand{\Re}{\text{\rm Re}}
\renewcommand{\Im}{\text{\rm Im}}
\renewcommand{\ln}{\text{\rm ln}}

\renewcommand{\le}{\leqslant}

\newtheorem{theorem}{Theorem}[section]

\newtheorem{lemma}[theorem]{Lemma}

\newtheorem{hypothesis}[theorem]{Hypothesis}
\theoremstyle{definition}
\newtheorem{definition}[theorem]{Definition}
\newtheorem{remark}[theorem]{Remark}

\allowdisplaybreaks \numberwithin{equation}{section}

\begin{document}

\title[Symmetrized Perturbation Determinants and Boundary Data Maps]
{Symmetrized Perturbation Determinants and Applications
to Boundary Data Maps and Krein-Type Resolvent
Formulas}

\author[F.\ Gesztesy and M.\ Zinchenko]{Fritz Gesztesy
and Maxim Zinchenko}

\address{Department of Mathematics,
University of Missouri, Columbia, MO 65211, USA}
\email{fritz@math.missouri.edu}
\urladdr{http://www.math.missouri.edu/personnel/faculty/gesztesyf.html}

\address{Department of Mathematics, Western Michigan University,
Kalamazoo, MI 49008, USA \newline
After August 1, 2010: Department of Mathematics,
University of Central Florida, Orlando, FL 32816, USA}
\email{maxim.zinchenko@wmich.edu}
\urladdr{http://www.z-max.info/}

\dedicatory{Dedicated to the memory of Pierre Duclos (1948--2010)}

\thanks{Based upon work partially supported by the US National Science
Foundation under Grant No.\ DMS-0965411.}




\date{\today}

\begin{abstract}
The aim of this paper is twofold: On one hand we discuss an abstract approach to symmetrized Fredholm perturbation determinants and an associated trace formula for a pair of operators of positive-type, extending a classical trace formula.

On the other hand, we continue a recent systematic study of boundary 
data maps in \cite{CGM10}, that is, $2 \times 2$ matrix-valued 
Dirichlet-to-Neumann and more generally, Robin-to-Robin maps, associated with one-dimensional Schr\"odinger operators on a compact interval $[0,R]$ with separated boundary conditions at $0$ and $R$. 
One of the principal new results in this paper reduces an appropriately symmetrized (Fredholm) perturbation determinant to the $2\times 2$ determinant of the underlying boundary data map. In addition, as a concrete application of the abstract approach in the first part of this 
paper, we establish the trace formula for resolvent differences of 
self-adjoint Schr\"odinger operators corresponding to different 
(separated) boundary conditions in terms of boundary data maps.
\end{abstract}

\maketitle

\section{Introduction}    \lb{s1}

In his joint 1983 paper  \cite{CDS83} with Jean-Michel Combes and Ruedi Seiler,
Pierre Duclos considered various one-dimensional Dirichlet and Neumann
Schr\"odinger operators and associated Krein-type resolvent formulas to study
the classical limit of discrete eigenvalues in a multiple-well potential. One of
the principal aims of the present paper is to consider related Krein-type
resolvent formulas for general separated boundary conditions on a compact
interval and establish connections with recently established boundary data
maps in \cite{CGM10}, perturbation determinants, and trace formulas. In
addition, we discuss an abstract approach to symmetrized (Fredholm)
perturbation determinants and an associated trace formula for a pair of operators
of positive-type, extending a classical trace formula for perturbation determinants
described by Gohberg and Krein \cite[Sect.\ IV.3]{GK69}.

In Section \ref{s7} we depart from our consideration of Schr\"odinger operators on
a compact interval and turn our attention to an abstract result on symmetrized
(Fredholm) determinants of the form
\begin{equation}
{\det}_{\cH}\Big(\ol{(A - z I_{\cH})^{1/2}(A_0 - z I_{\cH})^{-1}
	(A - z I_{\cH})^{1/2}}\Big)    \lb{1.7a}
\end{equation}
associated with a pair of operators $(A, A_0)$ of positive-type (and $z$ in
appropriate sectors of the complex plane). In particular, this
permits a discussion of sectorial (and hence non-self-adjoint) operators. It also
naturally permits a study of self-adjoint operators $(A, A_0)$, where $A$ is a small
form perturbation of $A_0$, extending the traditional case in which $A$ is a small
(Kato--Rellich-type) operator perturbation of $A_0$. Our principal result in
Section \ref{s7} then concerns a proof of the trace formula
\begin{align}
\begin{split}
& - \f{d}{dz} \ln\Big({\det}_{\cH}\Big(\ol{(A - z I_{\cH})^{1/2}(A_0 - z I_{\cH})^{-1}
	(A - z I_{\cH})^{1/2}}\Big)\Big)   \\
& \quad = {\tr}_{\cH}\big((A - z I_{\cH})^{-1} - (A_0 - z I_{\cH})^{-1}\big),   \lb{1.7c}
\end{split}
\end{align}
\noindent
an extension of the well-known operator perturbation case in which the symmetrized
expression
\begin{equation}
\ol{(A - z I_{\cH})^{1/2}(A_0 - z I_{\cH})^{-1} (A - z I_{\cH})^{1/2}}
\end{equation}
is replaced by the traditional expression
\begin{equation}
(A - z I_{\cH}) (A_0 - z I_{\cH})^{-1}
\end{equation}
on the left-hand side of \eqref{1.7c} (cf.\ Gohberg and Krein \cite[Sect.\ IV.3]{GK69}). The generalized trace formula \eqref{1.7c} appears to be without precedent under our general hypothesis that $A$ and $A_0$ are operators of positive-type and hence seems to be of independent interest.

Returning to the second principal aim of this paper, the discussion of boundary data maps for Schr\"odinger operators on a compact interval with separated boundary conditions, let $R>0$, introduce the strip
$S_{2 \pi}=\{z\in\bbC\,|\, 0\leq \Re(z) < 2 \pi\}$, and consider the boundary trace map
\begin{equation} \lb{1.1}
\gate \colon \begin{cases}
C^1({[0,R]}) \rightarrow \bbC^2, \\
u \mapsto \begin{bmatrix} \cos(\te_0)u(0) + \sin(\te_0)u'(0)\\
\cos(\te_R)u(R) - \sin(\te_R)u'(R) \end{bmatrix}, \end{cases}
\quad  \te_0, \te_R\in S_{2 \pi},
\end{equation}
where ``prime'' denotes $d/dx$. In addition, assuming that
\begin{equation}
V\in L^1((0,R); dx)    \lb{1.2}
\end{equation}
($V$ is not assumed to be real-valued in Sections \ref{s1} and \ref{s2}),
one can introduce the family of one-dimensional Schr\"odinger operators
$\Hte$ in $L^2((0,R); dx)$ by
\begin{align}
& \Hte f= -f'' + Vf,  \quad \te_0, \te_R\in S_{2 \pi},     \no   \\
& f\in\dom(\Hte)=\big\{ g \in L^2((0,R); dx)\,\big|\, g, g'\in AC({[0,R]}); \, \gate (g)=0;  \label{1.3} \\
& \hspace*{6.75cm} (-g''+Vg)\in L^2((0,R); dx)\big\},   \no
\end{align}
where $AC([0,R])$ denotes the set of absolutely continuous functions on $[0,R]$.

Assuming that $z\in\CR$ (with $\sigma(T)$ denoting the spectrum of $T$)
and $\te_0, \te_R \in S_{2 \pi}$, we recall that the boundary value problem given by
\begin{align}
-&u'' + Vu=zu,\quad u, u'\in AC([0,R]),   \label{1.4}\\
&\gate(u)=\begin{bmatrix}c_0\\ c_R \end{bmatrix}\in\bbC^2,    \label{1.5}
\end{align}
has a unique solution denoted by
$u(z,\cdot)=u(z,\cdot\,;(\te_0,c_0),(\te_R,c_R))$ for each $c_0, c_R\in\bbC$. To each boundary value problem \eqref{1.4}, \eqref{1.5}, we now associate a
family of \emph{general boundary data maps},
$\Lates (z) : \bbC^2 \rightarrow \bbC^2$, for
$\te_0, \te_R, \te_0^{\prime},\te_R^{\prime}\in S_{2 \pi}$,  where
\begin{align}\label{1.6}
\begin{split}
\Lates (z) \begin{bmatrix}c_0\\ c_R \end{bmatrix} &=
\Lates (z) \big(\gate(u(z,\cdot\ ;(\te_0,c_0),(\te_R,c_R)))\big)   \\
&= \gates(u(z,\cdot\ ;(\te_0,c_0),(\te_R,c_R))).
\end{split}
\end{align}
With $u(z,\cdot)=u(z,\cdot\ ;(\te_0,c_0),(\te_R,c_R))$, $\Lates (z) $ can be represented as a $2\times 2$ complex matrix, where
\begin{align}\label{1.7}
\Lates (z) \begin{bmatrix}c_0\\ c_R \end{bmatrix} &=
\Lates (z) \begin{bmatrix} \cos(\te_0)u(z,0) + \sin(\te_0)u'(z,0)\\[1mm]
\cos(\te_R)u(z,R) - \sin(\te_R)u'(z,R) \end{bmatrix}  \no \\
&=
\begin{bmatrix} \cos(\te_0^{\prime})u(z,0) + \sin(\te_0^{\prime})u'(z,0)\\[1mm]
\cos(\te_R^{\prime})u(z,R) - \sin(\te_R^{\prime})u'(z,R) \end{bmatrix}.
\end{align}

The map $\Lates (z)$, $z\in\CR$, was the principal object studied in the
recent paper \cite{CGM10}.

In Section \ref{s2} we recall the principal results of \cite{CGM10} most relevant to
the present investigation. More precisely, we review the basic properties of $\Lates (z)$,
and detail the explicit representation of the boundary data maps $\Lates (z)$ in
terms of the resolvent of the underlying Schr\"odinger operator $\Hte$. We discuss the associated boundary trace maps, associated linear fractional transformations relating
the boundary data maps $\Lates (z)$ and $\Lades(z)$ and mention the fact that
$\Lates (\cdot)$ is a matrix-valued Herglotz function (i.e., analytic on $\bbC_+$,
the open complex upper half-plane, with a nonnegative imaginary part) in the special case where $\Hte$ is self-adjoint. We conclude our review of \cite{CGM10} with
Krein-type resolvent formulas explicitely relating the resolvents of $\Hte$ and
$H_{\theta_0',\theta_R'}$.

In Section \ref{s8}, we focus on the second group of new results in this paper and
relate $\Lates (z)$ with the trace formula for the difference of resolvents of $\Hte$
and $H_{\theta_0',\theta_R'}$ and the underlying perturbation determinants. In this
context we will be assuming self-adjointness of $\Hte$ and
$H_{\theta_0',\theta_R'}$. More precisely, we will prove the following facts:
\begin{align}
& {\det}_{L^2((0,R); dx)}\Big(\ol{(H_{\theta_0',\theta_R'} - z I)^{1/2}
(\Hte - z I)^{-1} (H_{\theta_0',\theta_R'} - z I)^{1/2}}\Big)    \no \\
& \quad = \f{\sin(\theta_0) \sin(\theta_R)}{\sin(\theta_0') \sin(\theta_R')} \,
{\det}_{\bbC^2} \Big(\Lates (z)\Big),      \lb{1.8A} \\
& \qquad \;\, \te_0, \te_R \in [0, 2 \pi), \; \te_0', \te_R' \in (0, 2 \pi)\backslash\{\pi\}, \;
z\in\rho(\Hte),     \no
\end{align}
and
\begin{align}
\begin{split}
& \tr_{L^2((0,R); dx)}\big((H_{\theta_0',\theta_R'} -z I)^{-1}
- (\Hte -z I)^{-1}\big)    \\
& \quad = - \f{d}{dz} \ln\Big({\det}_{\bbC^2}\Big(\Lates (z)\Big)\Big), \quad
z\in\rho(\Hte) \cap \rho(H_{\theta_0',\theta_R'}).    \lb{1.8}
\end{split}
\end{align}

For classical as well as recent fundamental  literature on Weyl--Titchmarsh
operators (i.e., spectral parameter dependent
Dirichlet-to-Neumann maps, or more generally, Robin-to-Robin maps, resp.,
Poincar\'e--Steklov operators), relevant in the context
of boundary value spaces (boundary triples, etc.), we refer, for instance, to
\cite{AB09}-- \cite{DM95}, \cite{GM08}--\cite{GMZ07}, \cite{Gr08},
\cite[Ch.\ 13]{Gr09}, \cite{Po04}--\cite{PR09}, \cite{Ry07}, \cite{Ry09}, and
especially, to the extensive bibliography in \cite{CGM10}.

Finally, we briefly summarize some of the notation used in this paper: Let $\cH$ be
a separable complex Hilbert space, $(\cdot,\cdot)_{\cH}$ the scalar product in $\cH$
(linear in the second argument), and $I_{\cH}$ the identity operator in $\cH$.
Next, let $T$ be a linear operator mapping (a subspace of) a
Banach space into another, with $\dom(T)$ and $\ker(T)$ denoting the
domain and kernel (i.e., null space) of $T$.
The closure of a closable operator $S$ is denoted by $\ol S$.
The spectrum essential spectrum, discrete spectrum, and resolvent set
of a closed linear operator in $\cH$ will be denoted by $\sigma(\cdot)$.
$\sigma_{\rm ess}(\cdot)$, $\sigma_{\rm d}(\cdot)$, and $\rho(\cdot)$, respectively.
The Banach space of bounded linear operators on $\cH$ is
denoted by $\cB(\cH)$, the analogous notation $\cB(\cX_1,\cX_2)$,
will be used for bounded operators between two Banach spaces $\cX_1$ and
$\cX_2$.
The Banach space of compact operators
defined on $\cH$ is denoted by $\cB_{\infty}(\cH)$ and the $\ell^p$-based trace
ideals are denoted by $\cB_p(\cH)$, $p \geq 1$. The Fredholm determinant for
trace class perturbations of the identity in $\cH$ is denoted by ${\det}_{\cH}(\cdot)$,
the trace for trace class operators in $\cH$ will be denoted by ${\tr}_{\cH}(\cdot)$.


\section{Symmetrized Perturbation Determinants and Trace Formulas: \\
An Abstract Approach}  \label{s7}

In this section we present our first group of new results, the connection between appropriate perturbation determinants and trace formulas in an abstract setting. Throughout this section, $\cH$ denotes a complex, separable Hilbert space with
inner product $(\cdot,\cdot)_{\cH}$, and $I_{\cH}$ represents the identity operator
in $\cH$. For basic facts on trace ideals and infinite determinants we refer, for
instance, to \cite{GGK97}--\cite{GK69}, \cite{Si77}, and \cite{Si05}.

We start with the following classical result:

\begin{theorem} [\cite{GK69}, p.\ 163] \lb{t7.1}
Let $T(\cdot)$ be analytic in the $\cB_1(\cH)$-norm on some open set
$\Omega \subseteq \bbC$. Then
${\det}_{\cH}(I_{\cH} + T(\cdot))$ is analytic in $\Omega$ and
\begin{align}
\begin{split}
\f{d}{dz} \ln({\det}_{\cH}(I_{\cH} + T(z)))
= {\tr}_{\cH}\big((I_{\cH} + T(z))^{-1} T'(z)\big),&      \lb{7.1} \\
z \in \big\{\zeta \in \Omega \,\big|\, (I_{\cH} + T(\zeta))^{-1} \in \cB(\cH)\big\}.&
\end{split}
\end{align}
\end{theorem}

Next, we recall a classical special case in connection with standard perturbation
determinants (cf.\ \cite[Ch.\ IV]{GK69}):

\begin{theorem} [\cite{GK69}, Sect.\ IV.3, \cite{Ku61}] \lb{t7.2}
Assume that $A$ and $ A_0$ are densely
defined, closed, linear operators in $\cH$ satisfying
\begin{align}
& \dom(A_0) \subseteq \dom(A),    \lb{7.2} \\
& (A - z I_{\cH}) \big[(A - z I_{\cH})^{-1} - (A_0 - z I_{\cH})^{-1}\big] \in \cB_1(\cH)
\, \text{ for some}     \no \\
& \hspace*{3.45cm} \text{$($and hence for all\,$)$ } \,
z \in \rho(A) \cap \rho(A_0)\big).
\lb{7.3}
\end{align}
Then
\begin{align}
& - \f{d}{dz} \ln\big({\det}_{\cH}\big((A - z I_{\cH})(A_0 - z I_{\cH})^{-1}\big)\big)
= {\tr}_{\cH}\big((A - z I_{\cH})^{-1} - (A_0 - z I_{\cH})^{-1}\big),   \no \\
& \hspace*{8.5cm}   z \in \rho(A) \cap \rho(A_0).     \lb{7.4}
\end{align}
\end{theorem}
\begin{proof}
For completeness, and since we intend to extend this type of result to certain
quadratic form perturbations, we briefly sketch the proof of \eqref{7.4}. Pick
$z \in \rho(A) \cap \rho(A_0)$. Since assumption \eqref{7.3} is equivalent to
\begin{equation}
- (A - z I_{\cH}) \big[(A - z I_{\cH})^{-1} - (A_0 - z I_{\cH})^{-1}\big]
= (A - A_0) (A_0 - z I_{\cH})^{-1} \in \cB_1(\cH),        \lb{7.5}
\end{equation}
the identity
\begin{equation}
(A - z I_{\cH})(A_0 - z I_{\cH})^{-1} = I_{\cH} + (A - A_0) (A_0 - z I_{\cH})^{-1}
\end{equation}
shows that ${\det}_{\cH}\big((A - z I_{\cH})(A_0 - z I_{\cH})^{-1}\big)$ is well-defined
and analytic for $z \in \rho(A_0)$.
Incidentally, \eqref{7.5} also yields that if \eqref{7.3} is satisfied for some
$z \in \rho(A) \cap \rho(A_0)$, then it is satisfied for all
$z \in \rho(A) \cap \rho(A_0)$. An application of \eqref{7.1} and cyclicity of the trace (i.e., $\tr_{\cH}(ST) = \tr_{\cH}(TS)$ whenever 
$S, T \in \cB(\cH)$ with $ST, TS \in \cB_1(\cH)$) imply
\begin{align}
& - \f{d}{dz} \ln\big({\det}_{\cH}\big((A - z I_{\cH})(A_0 - z I_{\cH})^{-1}\big)\big) \no \\
& \quad = - {tr}_{\cH}\Big(\big\{(A - z I_{\cH})(A_0 - z I_{\cH})^{-1}\big\}^{-1}
\big[(A - A_0)(A_0 - z I_{\cH})^{-1}\big]'\Big)   \no \\
& \quad = - {tr}_{\cH}\Big(\big\{(A - z I_{\cH})(A_0 - z I_{\cH})^{-1}\big\}^{-1}
(A - A_0)(A_0 - z I_{\cH})^{-2}\Big)   \no \\
& \quad = - {tr}_{\cH}\Big((A_0 - z I_{\cH})^{-1}
\big\{(A - z I_{\cH})(A_0 - z I_{\cH})^{-1}\big\}^{-1}
(A - A_0)(A_0 - z I_{\cH})^{-1}\Big)   \no \\
& \quad = {tr}_{\cH}\Big((A_0 - z I_{\cH})^{-1}
\big\{(A - z I_{\cH})(A_0 - z I_{\cH})^{-1}\big\}^{-1}     \no \\
& \qquad \qquad \; \, \times (A - z I_{\cH}) 
\big[(A - z I_{\cH})^{-1} - (A_0 - z I_{\cH})^{-1}\big]\Big)   \no \\
& \quad = {tr}_{\cH}\Big((A_0 - z I_{\cH})^{-1}
\big\{(A - z I_{\cH})(A_0 - z I_{\cH})^{-1}\big\}^{-1}
\big\{(A - z I_{\cH}) (A_0 - z I_{\cH})^{-1}\big\}    \no \\
& \qquad \qquad \; \, \times (A_0 - z I_{\cH})
\big[(A - z I_{\cH})^{-1} - (A_0 - z I_{\cH})^{-1}\big]\Big)   \no \\
& \quad = {tr}_{\cH}\big((A - z I_{\cH})^{-1} - (A_0 - z I_{\cH})^{-1}\big).
\end{align}
\end{proof}

For an extension of Theorem \ref{t7.2}, applicable, in particular, to suitable
quadratic form perturbations, we briefly recall a few basic facts on operators
of positive-type
and their fractional powers. While this theory has been fully developed in
connection with complex Banach spaces, we continue to restrict ourselves here
to the case of complex, separable Hilbert spaces. For details on this theory we
refer, for instance, to \cite[Chs.\ 2, 3, 7]{Ha06}, \cite[Ch.\ 4]{KZPS76},
\cite[Ch.\ 4]{Lu09}, \cite[Chs.\ 1, 3--5]{MS01}, and \cite[Chs.\ 2, 16]{Ya10}.

\begin{definition} \lb{d7.3}
Let $A$ be a densely defined, closed, linear operator in $\cH$ and denote by
$S_{\omega} \subset \bbC$, $\omega \in [0,\pi)$, the open sector
\begin{equation}
S_{\omega} = \begin{cases} \{z\in\bbC \,|\, z \neq 0, \, |\arg(z)|<\omega\},
& \omega \in (0,\pi), \\
(0,\infty), & \omega = 0,
\end{cases}
\end{equation}
with vertex at $z=0$ along the positive real axis and opening angle $2 \omega$. \\
$(i)$ $A$ is said to be of {\it nonnegative-type} if
\begin{align}
\begin{split}
& (\alpha) \;\; (-\infty, 0) \subset \rho(A),    \lb{7.6} \\
& (\beta) \;\;
M(A) = \sup_{t > 0} \big\|t (A + t I_{\cH})^{-1}\big\|_{\cB(\cH)} < \infty.
\end{split}
\end{align}
$(ii)$ $A$ is said to be of {\it positive-type} if
\begin{align}
\begin{split}
& (\alpha) \;\; (-\infty, 0] \subset \rho(A),    \lb{7.7} \\
& (\beta) \;\;
M_A = \sup_{t \geq 0} \big\|(1 + t) (A + t I_{\cH})^{-1}\big\|_{\cB(\cH)} < \infty.
\end{split}
\end{align}
$(iii)$ $A$ is called {sectorial of angle $\omega \in [0,\pi)$}, denoted by
$A \in \Sect(\omega)$, if
\begin{align}
\begin{split}
& (\alpha) \;\; \sigma(A) \subseteq \ol{S_{\omega}},     \lb{7.8} \\
& (\beta) \text{ For all $\omega' \in (\omega,\pi)$, }  \,
M(A,\omega') = \sup_{z\in \bbC\backslash \ol{S_{\omega'}}}
\big\| z (A - z I_{\cH})^{-1}\big\|_{\cB(\cH)} < \infty.
\end{split}
\end{align}
$(iv)$ $A$ is called \it{quasi-sectorial of angle $\omega \in [0,\pi)$} if there exists
$t_0 \in\bbR$ such that $A + t_0 I_{\cH}$ is sectorial of angle $\omega \in [0,\pi)$.
In this context we introduce the shifted sector $- t_0 + S_{\omega}$, where
\begin{equation}
- t_0 + S_{\omega} = \begin{cases} \{z\in\bbC \,|\, z \neq - t_0, \,
|\arg(- t_0 + z)|<\omega\},
& \omega \in (0,\pi), \\
(- t_0,\infty), & \omega = 0.
\end{cases}
\end{equation}
\end{definition}

Next, we recall a number of useful facts: \\

\noindent
$\textbf{(I)}$ If $A$ is of nonnegative-type, then (cf., e.g., 
\cite[Proposition\ 2.1.1\,a)]{Ha06})
\begin{equation}
M(A) \geq 1 \, \text{ and } \, A \in \Sect(\pi - \arcsin(1/M(A))).   \lb{7.9}
\end{equation}
Moreover, if $A$ is of nonnegative-type (resp., of positive-type) then
\begin{equation}
A + t I_{\cH} \, \text{ is of nonnegative-type (resp., of positive-type) for all $t>0$.}
\lb{7.9a}
\end{equation}
If $A$ is of positive-type, then (cf., e.g., \cite[Lemma\ 4.2]{Lu09})
\begin{equation}
\{z\in\bbC \,|\, \Re(z) \leq 0, \, |\Im(z)| < (|\Re(z)| + 1)/M_A\} \cup
\{z\in\bbC \,|\, |z| < 1/M_A\} \subset \rho(A),    \lb{7.9aa}
\end{equation}
and for every $\omega_0 \in (0, \arctan(1/M_A))$, $r_0 \in (0,1/M_A)$, there exists
$M_0(A,\omega_0,r_0)>0$ such that
\begin{align}
\begin{split}
& \|(A - z I_{\cH})^{-1}\|_{\cB(\cH)} \leq \f{M_0(A,\omega_0,r_0)}{1 + |z|},    \\
& z \in
\{\zeta \in \bbC \,|\, \Re(\zeta) < 0, \, |\Im(\zeta)|/|\Re(\zeta)| \leq \tan(\omega_0)\}
\cup \{\zeta \in \bbC \,|\, |\zeta| \leq r_0\}.
\end{split}
\end{align}
$\textbf{(II)}$ If $A\in \Sect(\omega)$ for some $\omega \in [0,\pi)$ 
and $\ker(A) = \{0\}$, then
(cf., e.g., \cite[Proposition\ 2.1.1\,b)]{Ha06})
\begin{equation}
A^{-1} \in \Sect(\omega) \, \text{ and } \,
M\big(A^{-1}, \omega'\big) \leq M(A,\omega') + 1, \quad \omega' \in (\omega,\pi).
\lb{7.10}
\end{equation}
$\textbf{(III)}$ If $A\in \Sect(\omega)$ for some $\omega \in [0,\pi)$, 
then (cf., e.g., \cite[Proposition\ 2.1.1\,j)]{Ha06})
\begin{equation}
A^* \in \Sect(\omega) \, \text{ and } \,
M(A^*, \omega'\big) = M(A,\omega'), \quad \omega' \in (\omega,\pi).   \lb{7.11}
\end{equation}
$\textbf{(IV)}$ Suppose $A$ is of positive-type then (cf., e.g., 
\cite[p.\ 280]{KZPS76})
\begin{equation}
A^{-\alpha} = \f{\sin(\pi \alpha)}{\pi} \int_0^\infty dt \, t^{- \alpha} (A + t I_{\cH})^{-1}
\in \cB(\cH), \quad 0 < \Re(\alpha) < 1.    \lb{7.12}
\end{equation}
(In this context of bounded operators $A^{-\alpha}$, $0 <\alpha < 1$, and integrands bounded in norm by a Lebesgue integrable function, the integral in \eqref{7.12} and
in analogous situations in this section, is viewed as a norm convergent
Bochner integral.)
Moreover, $A^{-\alpha}$ has an analytic continuation to the strip
$0 < \Re(\alpha) < n + 1$, $n\in\bbN$, given by
\begin{align}
& A^{-\alpha} = \f{\sin(\pi \alpha)}{\pi} \f{n!}{(1-\alpha)(2-\alpha) \cdots (n-\alpha)}
\int_0^\infty dt \, t^{n - \alpha} (A + t I_{\cH})^{-n - 1} \in \cB(\cH),    \no \\
& \hspace*{8cm}  0 < \Re(\alpha) < n + 1.      \lb{7.13}
\end{align}
In particular,
\begin{align}
A^{-\alpha} = \f{\sin(\pi \alpha)}{\pi (1-\alpha)}
\int_0^\infty dt \, t^{1 - \alpha} (A + t I_{\cH})^{-2} \in \cB(\cH), \quad
0 < \Re(\alpha) < 2.     \lb{7.14}
\end{align}
We also note that if $A \in \Sect(\omega)$ and $ \alpha \in (0,1)$, then (cf., e.g.,
\cite[Remark\ 3.1.16]{Ha06}) $A^\alpha \in \Sect(\alpha \omega)$,
$M(A^\alpha) \leq M(A)$, and
\begin{align}
\begin{split}
(A^\alpha - z I_{\cH})^{-1} = \f{\sin(\pi \alpha)}{\pi} \int_0^{\infty} dt \,
\f{t^\alpha}{(z-t^\alpha e^{i \pi \alpha}) (z-t^\alpha e^{- i \pi \alpha})}
(A + t I_{\cH})^{-1},&   \\
|\arg(z)| > \alpha \pi.&
\end{split}
\end{align}
$\textbf{(V)}$ Suppose $A$ is of positive-type and $0 < \Re(\alpha) < n$ 
for some $n\in\bbN$, then (cf., e.g., \cite[Definition\ 4.5]{Lu09})
\begin{equation}
A^{\alpha} f = A^n A^{\alpha - n}f, \quad
f \in \dom(A^{\alpha}) = \{g \in \cH \,|\, A^{\alpha - n} g \in \dom(A^n)\}.    \lb{7.15}
\end{equation}
Moreover,
\begin{equation}
\dom(A^{\alpha}) = \ran(A^{-\alpha}) \, \text{ and } \, A^{\alpha} = (A^{-\alpha})^{-1},
\quad \Re(\alpha) > 0.    \lb{7.16}
\end{equation}
In particular, since $A^{-\alpha} \in \cB(\cH)$,
\begin{equation}
A^{\alpha} \, \text{ is closed in $\cH$ for all $\Re(\alpha) > 0$.}   \lb{7.17}
\end{equation}
$\textbf{(VI)}$ Suppose $A$ is of positive-type and $\Re(\alpha_1) > 0$, 
$\Re(\alpha_2) > 0$, then (cf., e.g., \cite[Proposition\ 4.4\,(iv)]{Lu09})
\begin{equation}
A^{-\alpha_1} A^{-\alpha_2} = A^{- \alpha_1 - \alpha_2}.    \lb{7.17a}
\end{equation}
$\textbf{(VII)}$ Suppose $A$ and $B$ are of positive-type and resolvent commuting, that is,
\begin{align}
\begin{split}
(A + s I_{\cH})^{-1} (B + t I_{\cH})^{-1} = (B + t I_{\cH})^{-1} (A + s I_{\cH})^{-1}& \\
\text{ for some (and hence for all) } \, s>0, \, t>0.&    \lb{7.17b}
\end{split}
\end{align}
Then (cf., e.g., \cite[p.\ 95]{Lu09})
\begin{align}
& (AB)^{\alpha} f = A^{\alpha} B^{\alpha} f = B^{\alpha} A^{\alpha} f = (BA)^{\alpha} f,
\no \\
& \quad f \in \dom((AB)^{\alpha})
= \{g \in \dom(B^{\alpha}) \,|\, B^{\alpha} g \in \dom(A^{\alpha})\}     \lb{7.17c} \\
& \qquad \;\;
= \{g \in \dom(A^{\alpha}) \,|\, A^{\alpha} g \in \dom(B^{\alpha})\}
= \dom((BA)^{\alpha}), \quad \alpha \in \bbC, \, \Re(\alpha) \neq 0.    \no
\end{align}
$\textbf{(VIII)}$ In the special case where $A$ is self-adjoint and strictly positive in $\cH$ (i.e., $A \geq \varepsilon I_{\cH}$ for some 
$\varepsilon > 0$), $A^\alpha$, $\alpha \in \bbC\backslash \{0\}$, defined on one hand as in the case of operators of positive-type above, and on 
the other by the spectral theorem, coincide (cf., e.g.,
\cite[Sect.\ 4.3.1]{Lu09}, \cite[Sect.\ 1.18.10]{Tr95}). In particular,
\begin{equation}
\dom(A^\alpha) = \bigg\{f\in\cH\,|\, \|A^{\alpha} f\|_{\cH}^2 =
\int_{[\varepsilon,\infty]} \lambda^{2 \Re(\alpha)} d\|E_A(\lambda) f\|_{\cH}^2
< \infty\bigg\}, \quad \alpha\in\bbC\backslash\{0\},
\end{equation}
in this case. Here $\{E_A(\lambda)\}_{\lambda \in \bbR}$ denotes the family of
spectral projections of $A$. \\
(It is possible to extend some of these formulas to $\Re(\alpha) = 0$, but we omit
the details since this will play no role in this manuscript.) \\

For the remainder of this section the basic assumptions on $A$ and $A_0$,
extending \eqref{7.2} and \eqref{7.3}, then read as follows:

\begin{hypothesis} \lb{h7.4}
Let $A$ and $A_0$ be densely defined, closed, linear operators in $\cH$. \\
$(i)$ Suppose there exists $t_0 \in \bbR$ such that $A + t_0 I_{\cH}$ and
$A_0 + t_0 I_{\cH}$ are of positive-type and $(A  + t_0 I_{\cH}) \in \Sect(\omega_0)$,
$(A_0  + t_0 I_{\cH}) \in \Sect(\omega_0)$ for some $\omega_0 \in [0,\pi)$. \\
$(ii)$ In addition, assume that for some $t_1 \geq t_0$,
\begin{align}
& \dom\big((A_0 + t_1 I_{\cH})^{1/2}\big) \subseteq
\dom\big((A + t_1 I_{\cH})^{1/2}\big),     \lb{7.18} \\
& \dom\big((A_0^* + t_1 I_{\cH})^{1/2}\big) \subseteq
\dom\big((A^* + t_1 I_{\cH})^{1/2}\big),     \lb{7.19} \\
& \ol{(A + t_1 I_{\cH})^{1/2} \big[(A + t_1 I_{\cH})^{-1} - (A_0 + t_1 I_{\cH})^{-1}\big]
	(A + t_1 I_{\cH})^{1/2}} \in \cB_1(\cH).    \lb{7.20}
\end{align}
\end{hypothesis}

One observes by item $(I)$, there always exists $\omega_0 \in [0,\pi)$ 
as in Hypothesis \ref{h7.4}\,$(i)$ as long as $A + t_0 I_{\cH}$ and
$A_0 + t_0 I_{\cH}$ are of nonnegative-type.

Our next results will show that if \eqref{7.18}--\eqref{7.20} hold for some
$t_1 \geq t_0$, then they actually extend to
$- t_1 = z \in \bbC \backslash \big(\ol{- t_0 + S_{\omega_0}}\big)$:

\begin{lemma} \lb{l7.5}
Assume that $A$ satisfy Hypothesis \ref{h7.4}\,$(i)$. Then
$(A + t I_{\cH})^{-1/2} \in \cB(\cH)$ $($resp., $(A^* + t I_{\cH})^{-1/2} \in \cB(\cH)$$)$,
$t>t_0$, analytically extends to $(A - z I_{\cH})^{-1/2} \in \cB(\cH)$ $($resp.,
$(A^* - z I_{\cH})^{-1/2} \in \cB(\cH)$$)$ for
$z \in \bbC \backslash \big(\ol{- t_0 + S_{\omega_0}}\big)$. In addition,
\begin{align}
\dom\big((A - z I_{\cH})^{1/2}\big) &= \dom\big((A + t_0 I_{\cH})^{1/2}\big),
\quad z \in \bbC \backslash \big(\ol{- t_0 + S_{\omega_0}}\big),    \lb{7.20Aa}  \\
\dom\big((A^* - z I_{\cH})^{1/2}\big) &= \dom\big((A^* + t_0 I_{\cH})^{1/2}\big),
\quad z \in \bbC \backslash \big(\ol{- t_0 + S_{\omega_0}}\big).    \lb{7.20Ab}
\end{align}
\end{lemma}
\begin{proof}
Applying \eqref{7.12} with $\alpha = 1/2$ and $A$ replaced by
$(A + s I_{\cH})$, $s > t_0$, one obtains
\begin{equation}
(A + s I_{\cH})^{-1/2} = \f{1}{\pi} \int_0^{\infty} dt \, t^{-1/2}
(A + (s + t) I_{\cH})^{-1}, \quad s > t_0.   \lb{7.20b}
\end{equation}
The resolvent estimates in \eqref{7.7} and \eqref{7.8} then prove that
$(A + s I_{\cH})^{-1/2}$, $s > t_0$, analytically extends to
$(A - z I_{\cH})^{-1/2} \in \cB(\cH)$,
$z \in \bbC \backslash \big(\ol{- t_0 + S_{\omega_0}}\big)$, with the result
\begin{equation}
(A - z I_{\cH})^{-1/2} = \f{1}{\pi} \int_0^{\infty} dt \, t^{-1/2}
(A + (- z + t) I_{\cH})^{-1}, \quad
z \in \bbC \backslash \big(\ol{- t_0 + S_{\omega_0}}\big).   \lb{7.20ba}
\end{equation}

In the following we choose $z,\, z_1 \in \bbC \backslash \big(\ol{- t_0 + S_{\omega_0}}\big)$ such that $|z_1-z| < \big\|(A - z_1 I_{\cH})^{-1}\big\|_{\cB(\cH)}^{-1}$ and consider the resolvent identity
\begin{equation}
(A - z I_{\cH}) = (A - z_1 I_{\cH})
\big[I_{\cH} + (z_1-z) (A - z_1 I_{\cH})^{-1}\big]. \lb{7.20AA}
\end{equation}
It follows from $(A  + t_0 I_{\cH}) \in \Sect(\omega_0)$ and \eqref{7.7}, \eqref{7.8} that 
\begin{equation} 
A - z I_{\cH}, \;  
z \in \bbC \backslash \big(\ol{- t_0 + S_{\omega_0}}\big), \, \text{ is of positive-type.}      \lb{7.20AC}
\end{equation} 
To prove the claim \eqref{7.20AC} we first note that 
$z \in \bbC \backslash \big(\ol{- t_0 + S_{\omega_0}}\big)$ implies that 
$(-\infty, 0] \subset \rho(A - z I_{\cH})$. Next, one chooses 
$\omega'_0 \in (\omega,\pi)$ such that actually,  
$z \in \bbC \backslash \big(\ol{- t_0 + S_{\omega'_0}}\big)$. Then 
\begin{align}
\begin{split} 
& \sup_{t\geq0} \big\|(1+t)(A + (t - z)I_{\cH})^{-1}\big\|_{\cB(\cH)} \\ 
& \quad = \sup_{\zeta=z+t_0-t, \, t\geq0}
|(1+t)/\zeta| \big\|\zeta (A+(t_0-z)I_{\cH})^{-1}\big\|_{\cB(\cH)} 
\leq C(z) <\infty
\end{split} 
\end{align} 
since $\sup_{t\geq 0} |(1+t)/(z + t_0 - t)| < \infty$, the estimate 
\eqref{7.7}\,$(\beta)$ becomes a special case of \eqref{7.8}\,$(\beta)$. 

In addition, $\big\|(z_1-z)(A - z_1 I_{\cH})^{-1}\big\|_{\cB(\cH)} < 1$ implies that $B=I_{\cH} + (z_1-z) (A - z_1 I_{\cH})^{-1}$ is of positive-type as well since
\begin{align}
& \left\|(1+t)(B+tI_{\cH})^{-1}\right\|_{\cB(\cH)} 
= \left\|(1+t)\left[(1+t)I_{\cH} 
+ (z_1-z) (A - z_1 I_{\cH})^{-1}\right]^{-1}\right\|_{\cB(\cH)} \no
\\
& \quad = \bigg\|\Big[I_{\cH} + \f{z_1-z}{1+t}(A - z_1 I_{\cH})^{-1}\Big]^{-1}\bigg\|_{\cB(\cH)}     \no \\[1mm] 
& \quad \leq \Big[1- \left\|(z_1-z)(A - z_1 I_{\cH})^{-1}\right\|_{\cB(\cH)} 
\Big]^{-1}<\infty,  \quad  t\geq0.   \lb{7.20AB}
\end{align}
Thus $(VII)$ applied to the resolvent identity \eqref{7.20AA} yields
\begin{align}
\begin{split}
(A - z I_{\cH})^{-1/2} = (A - z_1 I_{\cH})^{-1/2}
\big[I_{\cH} + (z_1 - z) (A - z_1 I_{\cH})^{-1}\big]^{-1/2},&    \\
z, z_1 \in \bbC \backslash \big(\ol{- t_0 + S_{\omega_0}}\big), \;
|z-z_1| < \big\|(A - z_1 I_{\cH})^{-1}\big\|_{\cB(\cH)}^{-1}.&
\end{split}
\lb{7.20cd}
\end{align}
Bounded invertibility of $\big[I_{\cH} + (-z + z_1) (A - z_1 I_{\cH})^{-1}\big]^{-1/2}$
for $z, z_1 \in \bbC \backslash \big(\ol{- t_0 + S_{\omega_0}}\big)$,
$|z-z_1| < \big\|(A - z_1 I_{\cH})^{-1}\big\|_{\cB(\cH)}^{-1}$ then implies that
$\ran\big((A - z I_{\cH})^{-1/2}\big)$ is locally constant in
$z \in \bbC \backslash \big(\ol{- t_0 + S_{\omega_0}}\big)$,
\begin{align}
\begin{split}
& \ran\big((A - z I_{\cH})^{-1/2}\big) = \ran\big((A -z_1 I_{\cH})^{-1/2}\big), \\
& z, z_1 \in \bbC \backslash \big(\ol{- t_0 + S_{\omega_0}}\big), \;
|z-z_1| < \big\|(A - z_1 I_{\cH})^{-1}\big\|_{\cB(\cH)}^{-1},     \lb{7.20ce}
\end{split}
\end{align}
and hence that
\begin{equation}
\ran\big((A - z I_{\cH})^{-1/2}\big) = \ran\big((A -z_1 I_{\cH})^{-1/2}\big),  \quad
z, z_1 \in \bbC \backslash \big(\ol{- t_0 + S_{\omega_0}}\big).     \lb{7.20cf}
\end{equation}
An application of \eqref{7.16} then gives \eqref{7.20Aa}. Equation
\eqref{7.20Ab} is proved analogously with the help of $(III)$.
\end{proof}

\begin{lemma} \lb{l7.6}
Assume that $A$ and $A_0$ satisfy Hypothesis \ref{h7.4}\,$(i)$ and suppose
that \eqref{7.18} and \eqref{7.19} hold for some $t_1 \geq t_0$. Then \eqref{7.18}
and \eqref{7.19} extend to
\begin{align}
& \dom\big((A_0 - z I_{\cH})^{1/2}\big) \subseteq
\dom\big((A - z I_{\cH})^{1/2}\big), \quad
z \in \bbC \backslash \big(\ol{- t_0 + S_{\omega_0}}\big),    \lb{7.20da} \\
& \dom\big((A_0^* - z I_{\cH})^{1/2}\big) \subseteq
\dom\big((A^* - z I_{\cH})^{1/2}\big),
\quad z \in \bbC \backslash \big(\ol{- t_0 + S_{\omega_0}}\big). \lb{7.20db}
\end{align}
Moreover,
\begin{align}
& (A - z I_{\cH})^{1/2} (A_0 - z I_{\cH})^{-1/2} \in \cB(\cH) \, \text{ and } \,
(A^* - z I_{\cH})^{1/2} (A_0^* - z I_{\cH})^{-1/2} \in \cB(\cH),    \no  \\
& \quad   \text{ are analytic for
$z \in \bbC \backslash \big(\ol{- t_0 + S_{\omega_0}}\big)$
with respect to the $\cB(\cH)$-norm.}    \lb{7.20f}
\end{align}
\end{lemma}
\begin{proof}
By items $(I)$ and $(III)$ it again suffices to just focus on the proof of 
\eqref{7.20da}.

Since by \eqref{7.20Aa} and \eqref{7.20Ab} the domains of $(A - z I_{\cH})^{1/2}$
and $(A^* - z I_{\cH})^{1/2}$ are $z$-independent for
$z \in \bbC \backslash \big(\ol{- t_0 + S_{\omega_0}}\big)$, \eqref{7.18} and
\eqref{7.19} extend to $z \in \bbC \backslash \big(\ol{- t_0 + S_{\omega_0}}\big)$.

To prove the analyticity statement involving $A$ and $A_0$ in \eqref{7.20f} we write
\begin{align}
& (A - z I_{\cH})^{1/2} (A_0 - z I_{\cH})^{-1/2} =
\big[(A - z I_{\cH})^{1/2} (A - z_0 I_{\cH})^{-1/2}\big]    \no \\
& \quad \times \big[(A - z_0 I_{\cH})^{1/2} (A_0 - z_0 I_{\cH})^{-1/2}\big]
\big[(A_0 - z_0 I_{\cH})^{1/2} (A_0 - z I_{\cH})^{-1/2}\big],    \lb{7.20g} \\
& \hspace*{7.33cm}
z, z_0 \in \bbC \backslash \big(\ol{- t_0 + S_{\omega_0}}\big),   \no
\end{align}
and separately investigate each of the three factors in \eqref{7.20g}. Since by
hypothesis \eqref{7.20Aa} holds for $A$ and $A_0$, \eqref{7.20da} yields that
\begin{equation}
(A - z_0 I_{\cH})^{1/2} (A_0 - z_0 I_{\cH})^{-1/2} \in \cB(\cH). \lb{7.20ga}
\end{equation}
Next, applying \eqref{7.20cd} with $A$ replaced by $A_0$ yields
\begin{align}
\begin{split}
&(A_0 - z_0 I_{\cH})^{1/2} (A_0 - z I_{\cH})^{-1/2}
\\
&\quad = \big[(A_0 - z_0 I_{\cH})^{1/2}(A_0 - z_1 I_{\cH})^{-1/2}\big] \big[I_{\cH} + (z_1-z) (A_0 - z_1 I_{\cH})^{-1}\big]^{-1/2},
\\
&\hspace{2.35cm} z, z_0, z_1 \in \bbC \backslash \big(\ol{- t_0 + S_{\omega_0}}\big), \; |z-z_1| < \big\|(A - z_1 I_{\cH})^{-1}\big\|_{\cB(\cH)}^{-1}.
\end{split} \lb{7.20h}
\end{align}
Since by \eqref{7.20AB} (with $A$ replaced by $A_0$) $B=I_{\cH} + (z_1-z) (A_0 - z_1 I_{\cH})^{-1}$ is of positive-type, it follows from \eqref{7.12} (with $\alpha = 1/2$ and $A$ replaced by $B$) and a geometric series expansion that
\begin{align}
B^{-1/2} &= \big[I_{\cH} + (z_1-z) (A_0 - z_1 I_{\cH})^{-1}\big]^{-1/2}   \no \\
&= \f{1}{\pi}
\int_0^{\infty} dt \, t^{-1/2} \big[(1+t) I_{\cH} + (z_1-z) (A_0 - z_1 I_{\cH})^{-1}\big]^{-1}
\no \\
&= \sum_{m=0}^{\infty} (-1)^m (z_1-z)^m (A_0 - z_1 I_{\cH})^{-m} \,
 \f{1}{\pi} \int_0^{\infty} \f{dt \, t^{-1/2}}{(1 + t)^{m+1}}  \no \\
&= \sum_{m=0}^{\infty} \f{\Gamma(m + (1/2))}{\Gamma(m+1) \Gamma(1/2)}
(-1)^m (z_1-z)^m (A_0 - z_1 I_{\cH})^{-m}.   \lb{7.20a}
\end{align}
Thus $\big[I_{\cH} + (z_1-z) (A_0 - z_1 I_{\cH})^{-1}\big]^{-1/2}$ is analytic with respect to $z$ for $z,\, z_1 \in \bbC \backslash \big(\ol{- t_0 + S_{\omega_0}}\big)$, $|z-z_1| < \big\|(A_0 - z_1 I_{\cH})^{-1}\big\|_{\cB(\cH)}^{-1}$. Moreover, by \eqref{7.20Aa} (with $A$ replaced by $A_0$), $(A_0 - z_0 I_{\cH})^{1/2}(A_0 - z_1 I_{\cH})^{-1/2}\in\cB(\cH)$, and hence one concludes that the left-hand side of \eqref{7.20h} is analytic with respect to $z \in \bbC \backslash \big(\ol{- t_0 + S_{\omega_0}}\big)$.

Finally, writing
\begin{align}
& (A - z I_{\cH})^{1/2} (A - z_0 I_{\cH})^{-1/2} =
(A - z I_{\cH}) (A - z I_{\cH})^{-1/2} (A - z_0 I_{\cH})^{-1/2}   \no \\
& \quad = A (A - z I_{\cH})^{-1/2} (A - z_0 I_{\cH})^{-1/2}
- z (A - z I_{\cH})^{-1/2} (A - z_0 I_{\cH})^{-1/2}    \lb{7.20i}
\end{align}
it suffices to focus on the term $A (A - z I_{\cH})^{-1/2} (A - z_0 I_{\cH})^{-1/2}$.
Writing
\begin{align}
\begin{split}
& A (A - z I_{\cH})^{-1/2} (A - z_0 I_{\cH})^{-1/2}   \\
& \quad = \big[A (A - z_0 I_{\cH})^{-1}\big]
\big[(A - z_0 I_{\cH})^{1/2} (A - z I_{\cH})^{-1/2}\big],    \lb{7.20j}
\end{split}
\end{align}
employing the obvious fact that $A (A - z_0 I_{\cH})^{-1} \in \cB(\cH)$, the analyticity
of $(A - z_0 I_{\cH})^{1/2} (A - z I_{\cH})^{-1/2}$ (and hence that of the left-hand
sides in \eqref{7.20i} and \eqref{7.20j}) with respect to
$z \in \bbC \backslash \big(\ol{- t_0 + S_{\omega_0}}\big)$ then follows
as in \eqref{7.20h} above with $A_0$ replaced by $A$.
\end{proof}

\begin{lemma} \lb{l7.7}
Assume that $A$ and $A_0$ satisfy Hypothesis \ref{h7.4}. Then \eqref{7.20}
extends to
\begin{align}
\begin{split}
\ol{(A -z I_{\cH})^{1/2} \big[(A - z I_{\cH})^{-1} - (A_0 - z I_{\cH})^{-1}\big]
	(A - z I_{\cH})^{1/2}} \in \cB_1(\cH),&    \lb{7.20k} \\
z \in \bbC \backslash \big(\ol{- t_0 + S_{\omega_0}}\big).& 	
\end{split}
\end{align}
In addition, $\ol{(A -z I_{\cH})^{1/2}
\big[(A - z I_{\cH})^{-1} - (A_0 - z I_{\cH})^{-1}\big] (A - z I_{\cH})^{1/2}}$ is
analytic for $z \in \bbC \backslash \big(\ol{- t_0 + S_{\omega_0}}\big)$ with
respect to the $\cB_1(\cH)$-norm.
\end{lemma}
\begin{proof} Let $z \in \bbC \backslash \big(\ol{- t_0 + S_{\omega_0}}\big)$
and $t_1 \geq t_0$ as in \eqref{7.20}. Using the fact
\begin{align}
\begin{split}
& (A - z I_{\cH})^{-1} - (A_0 - z I_{\cH})^{-1}
= (A + t_1 I_{\cH}) (A - z I_{\cH})^{-1}    \\
& \quad \times \big[(A + t_1 I_{\cH})^{-1} - (A_0 + t_1 I_{\cH})^{-1}\big]
(A_0 + t_1 I_{\cH}) (A_0 - z I_{\cH})^{-1},
\end{split}
\end{align}
one obtains
\begin{align}
& \ol{(A -z I_{\cH})^{1/2} \big[(A - z I_{\cH})^{-1} - (A_0 - z I_{\cH})^{-1} \big]
(A - z I_{\cH})^{1/2}}   \no \\
& \quad = \big[(A -z I_{\cH})^{1/2} (A + t_1 I_{\cH})^{-1/2}\big]
\big[(A + t_1 I_{\cH}) (A - z I_{\cH})^{-1}\big]    \no \\
& \qquad \times \cl\big\{(A + t_1 I_{\cH})^{1/2}
\big[(A + t_1 I_{\cH})^{-1} - (A_0 + t_1 I_{\cH})^{-1}\big] (A + t_1 I_{\cH})^{1/2}
\no \\
& \qquad \times (A + t_1 I_{\cH})^{-1/2}
(A_0 + t_1 I_{\cH}) (A_0 - z I_{\cH})^{-1} (A - z I_{\cH})^{1/2}\big\}   \no \\
& \quad = \big[(A -z I_{\cH})^{1/2} (A + t_1 I_{\cH})^{-1/2}\big]
\big[(A + t_1 I_{\cH}) (A - z I_{\cH})^{-1}\big]    \no \\
& \qquad \times \ol{(A + t_1 I_{\cH})^{1/2}
\big[(A + t_1 I_{\cH})^{-1} - (A_0 + t_1 I_{\cH})^{-1}\big] (A + t_1 I_{\cH})^{1/2}}
\no \\
& \qquad \times \big\{(A -z I_{\cH})^{1/2} (A + t_1 I_{\cH})^{-1/2}    \lb{7.20l} \\
& \qquad \quad + (z+t) (A + t_1 I_{\cH})^{-1/2} (A - z I_{\cH})^{-1/2}
\big[(A^* - {\ol z} I_{\cH})^{1/2} (A_0^* - {\ol z} I_{\cH})^{-1/2}\big]^*\big\},   \no
\end{align}
where we employed the identity
\begin{equation}
(A_0 + t_1 I_{\cH}) (A_0 -z I_{\cH})^{-1}
= I_{\cH} + (z + t_1) (A_0 -z I_{\cH})^{-1}
\end{equation}
and used the symbol $\cl\{\dots\}$ to denote the operator closure (in addition to
our usual bar symbol) as the latter extends over two lines. By Lemmas
\ref{l7.6} and \ref{l7.7}, all square brackets $[\cdots]$ in \eqref{7.20l} lie in
$\cB(\cH)$. Thus, the trace class property in assumption \eqref{7.20} proves
that in \eqref{7.20k}.

Finally, the analyticity statements in \eqref{7.20f} (see also the one in
\eqref{7.20i}) employed in \eqref{7.20l} prove the $\cB_1(\cH)$-analyticity of
the operator in \eqref{7.20k}.
\end{proof}

\begin{theorem} \lb{t7.8}
Assume that $A$ and $A_0$ satisfy Hypothesis \ref{h7.4}.Then
\begin{align}
\begin{split}
& - \f{d}{dz} \ln\Big({\det}_{\cH}\Big(\ol{(A - z I_{\cH})^{1/2}(A_0 - z I_{\cH})^{-1}
	(A - z I_{\cH})^{1/2}}\Big)\Big)   \\
&\quad = {\tr}_{\cH}\big((A - z I_{\cH})^{-1} - (A_0 - z I_{\cH})^{-1}\big),     \lb{7.21}
\end{split}
\end{align}
for all $z \in \bbC \backslash \big(\ol{- t_0 + S_{\omega_0}}\big)$ such that $\ol{(A - z I_{\cH})^{1/2}(A_0 - z I_{\cH})^{-1}(A - z I_{\cH})^{1/2}}$ is boundedly invertible.
\end{theorem}
\begin{proof} Let $z\in \bbC \backslash \big(\ol{- t_0 + S_{\omega_0}}\big)$.
We note that by \eqref{7.20f} one has
\begin{align}
& \ol{(A - z I_{\cH})^{1/2} \big[(A - z I_{\cH})^{-1} - (A_0 - z I_{\cH})^{-1}\big]
	(A - z I_{\cH})^{1/2}}     \no \\
	& \quad = I_{\cH} - \ol{(A - z I_{\cH})^{1/2} (A_0 - z I_{\cH})^{-1}	
(A - z I_{\cH})^{1/2}}       \lb{7.22} \\
& \quad = I_{\cH} - \big[(A - z I_{\cH})^{1/2} (A_0 - z I_{\cH})^{-1/2}\big]	
\big[(A^* - {\ol z} I_{\cH})^{1/2} (A_0^* - {\ol z} I_{\cH})^{-1/2}\big]^*   \no
\end{align}
and hence ${\det}_{\cH}\big(\ol{(A - z I_{\cH})^{1/2}(A_0 - z I_{\cH})^{-1}
(A - z I_{\cH})^{1/2}}\big)$ is well-defined.

Next, we consider
\begin{equation}
T_1(z) = (A - z I_{\cH})^{1/2} (A_0 - z I_{\cH})^{-1/2}, \quad
z \in \bbC \backslash \big(\ol{- t_0 + S_{\omega_0}}\big),
\end{equation}
and compute for $\varepsilon \in \bbC\backslash\{0\}$, $|\varepsilon|$
sufficiently small such that
$z, z+\varepsilon \in \bbC \backslash \big(\ol{- t_0 + S_{\omega_0}}\big)$,
\begin{align}
& [T_1(z+\varepsilon) - T_1(z)]     \no \\
& \quad = (A - (z+\varepsilon) I_{\cH}) (A - (z+\varepsilon) I_{\cH})^{-1/2}
(A_0 - (z+\varepsilon) I_{\cH})^{-1/2}    \no \\
& \qquad - (A - z I_{\cH}) (A - z I_{\cH})^{-1/2} (A_0 - z I_{\cH})^{-1/2}    \no \\
& \quad = (A - (z+\varepsilon) I_{\cH}) (A - (z+\varepsilon) I_{\cH})^{-1/2}
(A_0 - (z+\varepsilon) I_{\cH})^{-1/2}    \no \\
& \qquad - (A - z I_{\cH}) (A - (z+\varepsilon) I_{\cH})^{-1/2}
(A_0 - (z+\varepsilon) I_{\cH})^{-1/2}    \no \\
& \qquad + (A - z I_{\cH}) (A - (z+\varepsilon) I_{\cH})^{-1/2}
(A_0 - (z+\varepsilon) I_{\cH})^{-1/2}    \no \\
& \qquad - (A - z I_{\cH}) (A - z I_{\cH})^{-1/2} (A_0 - z I_{\cH})^{-1/2}    \no \\
& \quad = - \varepsilon (A - (z+\varepsilon) I_{\cH})^{-1/2}
(A_0 - (z+\varepsilon) I_{\cH})^{-1/2}     \no \\
& \qquad + (A - z I_{\cH}) (A - (z+\varepsilon) I_{\cH})^{-1/2}
(A_0 - (z+\varepsilon) I_{\cH})^{-1/2}    \no \\
& \qquad - (A - z I_{\cH}) (A - z I_{\cH})^{-1/2}
(A_0 - (z+\varepsilon) I_{\cH})^{-1/2}    \no \\
& \qquad + (A - z I_{\cH}) (A - z I_{\cH})^{-1/2}
(A_0 - (z+\varepsilon) I_{\cH})^{-1/2}    \no \\
& \qquad - (A - z I_{\cH}) (A - z I_{\cH})^{-1/2} (A_0 - z I_{\cH})^{-1/2}    \no \\
& \quad = - \varepsilon (A - (z+\varepsilon) I_{\cH})^{-1/2}
(A_0 - (z+\varepsilon) I_{\cH})^{-1/2}    \no \\
& \qquad + (A - z I_{\cH}) \big[(A - (z+\varepsilon) I_{\cH})^{-1/2} -
(A - z I_{\cH})^{-1/2}\big]     \no \\
& \qquad \quad \times (A_0 - (z+\varepsilon) I_{\cH})^{-1/2}    \lb{7.22a} \\
& \qquad + (A - z I_{\cH}) (A - z I_{\cH})^{-1/2}
\big[(A_0 - (z+\varepsilon) I_{\cH})^{-1/2} -
(A_0 - z I_{\cH})^{-1/2}\big].      \no
\end{align}
Using
\begin{align}
& \big[(A - (z+\varepsilon) I_{\cH})^{-1/2} - (A - z I_{\cH})^{-1/2}\big]   \no \\
& \quad = \f{1}{\pi} \int_0^{\infty} dt\, t^{-1/2}
\big[(A + (t - z - \varepsilon) I_{\cH})^{-1} - (A + (t - z) I_{\cH})^{-1}\big]   \no \\
& \quad = \f{\varepsilon}{\pi} \int_0^{\infty} dt \, t^{-1/2}
(A + (t - z - \varepsilon) I_{\cH})^{-1} (A + (t - z) I_{\cH})^{-1}
\end{align}
in \eqref{7.22a} yields
\begin{align}
& \f{1}{\varepsilon} [T_1(z+\varepsilon) - T_1(z)] =
- (A - (z+\varepsilon) I_{\cH})^{-1/2} (A_0 - (z+\varepsilon) I_{\cH})^{-1/2}     \no \\
& \qquad \quad + (A - z I_{\cH}) \bigg[\f{1}{\pi} \int_0^{\infty} dt \, t^{-1/2}
(A + (t - z - \varepsilon) I_{\cH})^{-1} (A + (t - z) I_{\cH})^{-1}\bigg]    \no \\
& \qquad \qquad \times (A_0 - (z+\varepsilon) I_{\cH})^{-1/2}     \no \\
& \quad \qquad + (A - z I_{\cH}) (A - z I_{\cH})^{-1/2}     \no \\
& \qquad \qquad \times \bigg[\f{1}{\pi} \int_0^{\infty} dt \, t^{-1/2}
(A_0 + (t - z - \varepsilon) I_{\cH})^{-1} (A_0 + (t - z) I_{\cH})^{-1}\bigg]    \no \\
& \quad \underset{\varepsilon \to 0}{\longrightarrow}
- (A - z I_{\cH})^{-1/2} (A_0 - z I_{\cH})^{-1/2}     \no \\
& \qquad \quad + (A - z I_{\cH}) \bigg[\f{1}{\pi} \int_0^{\infty} dt \, t^{-1/2}
(A + (t - z) I_{\cH})^{-2}\bigg]  (A_0 - z I_{\cH})^{-1/2}     \no \\
& \qquad \quad + (A - z I_{\cH}) (A - z I_{\cH})^{-1/2}
\bigg[\f{1}{\pi} \int_0^{\infty} dt \, t^{-1/2} (A_0 + (t - z) I_{\cH})^{-2}\bigg]    \no \\
& \quad = - (A - z I_{\cH})^{-1/2} (A_0 - z I_{\cH})^{-1/2}     \no \\
& \qquad + \f{1}{2} (A - z I_{\cH}) (A - z I_{\cH})^{-3/2} (A_0 - z I_{\cH})^{-1/2}  \no \\
& \qquad + \f{1}{2} (A - z I_{\cH})^{1/2} (A_0 - z I_{\cH})^{-3/2}    \no \\
& \quad = - \f{1}{2} (A - z I_{\cH})^{-1/2} (A_0 - z I_{\cH})^{-1/2}     \no \\
& \qquad + \f{1}{2} (A - z I_{\cH})^{1/2} (A_0 - z I_{\cH})^{-3/2},    \no \\
 \end{align}
where the limit $\varepsilon \to 0$ is valid in the $\cB(\cH)$-norm.
Here we used (cf.\ \eqref{7.14} applied to $\alpha = 3/2$ and $A$ replaced by
$(A - z I_{\cH})$)
\begin{equation}
\f{1}{\pi} \int_0^{\infty} dt \, t^{-1/2} (A + (t - z) I_{\cH})^{-2}
= \f{1}{2} (A - z I_{\cH})^{-3/2}, \quad
z \in \bbC \backslash \big(\ol{- t_0 + S_{\omega_0}}\big).
\end{equation}
Thus,
\begin{align}
& T_1'(z) = - \f{1}{2} (A - z I_{\cH})^{-1/2} (A_0 - z I_{\cH})^{-1/2}
+ \f{1}{2} (A- z I_{\cH})^{1/2} (A_0 - z I_{\cH})^{-3/2},    \no \\
& \hspace*{8cm}  z \in \bbC \backslash \big(\ol{- t_0 + S_{\omega_0}}\big).
\end{align}
Similarly, introducing
\begin{equation}
T_2(z) = \big[(A^* - {\ol z} I_{\cH})^{1/2} (A_0^* - {\ol z} I_{\cH})^{-1/2}\big]^*, \quad
z \in \bbC \backslash \big(\ol{- t_0 + S_{\omega_0}}\big),
\end{equation}
one obtains
\begin{align}
\begin{split}
T_2'(z) &= - \f{1}{2} \big[(A^* - {\ol z} I_{\cH})^{-1/2} (A_0 - z I_{\cH})^{-1/2}\big]^*   \\
& \quad
+ \f{1}{2} \big[(A^* - {\ol z} I_{\cH})^{1/2} (A_0^* - {\ol z} I_{\cH})^{-3/2}\big]^*,
\quad z \in \bbC \backslash \big(\ol{- t_0 + S_{\omega_0}}\big).
\end{split}
\end{align}
Consequently, one computes
\begin{align}
& \f{d}{dz}
\Big[\ol{(A - z I_{\cH})^{1/2}(A_0 - z I_{\cH})^{-1} (A - z I_{\cH})^{1/2}}\Big] \no \\
& \quad = [T_1(z) T_2(z)]' = T_1'(z) T_2(z) + T_1(z) T_2'(z)    \no \\
& \quad = \Big[- \f{1}{2} (A - z I_{\cH})^{-1/2} (A_0 - z I_{\cH})^{-1/2}
+ \f{1}{2} (A - z I_{\cH})^{1/2} (A_0 - z I_{\cH})^{-3/2}\Big]    \no \\
& \qquad \times \big[(A^* - {\ol z} I_{\cH})^{1/2}
(A_0^* - {\ol z} I_{\cH})^{-1/2}\big]^*   \no \\
& \qquad + \big[(A - z I_{\cH})^{1/2} (A_0 - z I_{\cH})^{-1/2}\big]
\Big[- \f{1}{2} \big[(A^* - {\ol z} I_{\cH})^{-1/2} (A_0^* - {\ol z} I_{\cH})^{-1/2}\big]^*   \no \\
& \qquad \quad
+ \f{1}{2} \big[(A^* - {\ol z} I_{\cH})^{1/2} (A_0^* - {\ol z} I_{\cH})^{-3/2}\big]^*\Big]
\no \\
& \quad = - \f{1}{2} \ol{(A - z I_{\cH})^{-1/2} (A_0 - z I_{\cH})^{-1} (A- z I_{\cH})^{1/2}}
\no \\
& \qquad - \f{1}{2} \ol{(A - z I_{\cH})^{1/2} (A_0 - z I_{\cH})^{-1} (A- z I_{\cH})^{-1/2}}
\no \\
& \qquad + \ol{(A - z I_{\cH})^{1/2} (A_0 - z I_{\cH})^{-2} (A- z I_{\cH})^{1/2}},
\quad z \in \bbC \backslash \big(\ol{- t_0 + S_{\omega_0}}\big).    \lb{7.22b}
\end{align}

Due to the $\cB_1(\cH)$-analyticity of the left-hand side of \eqref{7.22}
according to Lemma \ref{l7.7}, one can apply \eqref{7.1}, and using
the result \eqref{7.22b} one finally obtains
\begin{align}
& - \f{d}{dz} \ln\Big({\det}_{\cH}\Big(\ol{(A - z I_{\cH})^{1/2}(A_0 - z I_{\cH})^{-1}
(A - z I_{\cH})^{1/2}}\Big)\Big)    \no \\
& \quad = - {\tr}_{\cH} \Big(\big\{\ol{(A - z I_{\cH})^{1/2}(A_0 - z I_{\cH})^{-1}
(A - z I_{\cH})^{1/2}}\big\}^{-1}    \no \\
& \qquad \qquad \quad \, \times \big[\ol{(A - z I_{\cH})^{1/2}(A_0 - z I_{\cH})^{-1}
(A - z I_{\cH})^{1/2}}\big]'\Big)   \no \\
& \quad = \f{1}{2} {\tr}_{\cH} \Big(\big\{\ol{(A - z I_{\cH})^{1/2}(A_0 - z I_{\cH})^{-1}
(A - z I_{\cH})^{1/2}}\big\}^{-1}    \no \\
& \qquad \qquad \quad \, \times
\big\{\ol{(A - z I_{\cH})^{-1/2}(A_0 - z I_{\cH})^{-1} (A - z I_{\cH})^{1/2}}
\no \\
& \qquad \qquad \qquad \;\;\,
- 2 \ol{(A - z I_{\cH})^{1/2}(A_0 - z I_{\cH})^{-2} (A - z I_{\cH})^{1/2}} \no \\
& \qquad \qquad \qquad \;\;\,
+ \ol{(A - z I_{\cH})^{1/2}(A_0 - z I_{\cH})^{-1} (A - z I_{\cH})^{-1/2}}\big\}\Big) \no \\
& \quad = \f{1}{2} {\tr}_{\cH} \Big(\big\{\ol{(A - z I_{\cH})^{1/2}(A_0 - z I_{\cH})^{-1}
(A - z I_{\cH})^{1/2}}\big\}^{-1}    \no \\
& \qquad \qquad \quad \, \times
(A - z I_{\cH})^{1/2}\big[(A - z I_{\cH})^{-1} - (A_0 - z I_{\cH})^{-1}\big]   \no \\
& \qquad \qquad \quad \, \times \ol{(A_0 - z I_{\cH})^{-1} (A - z I_{\cH})^{1/2}}\Big)
\no \\
& \qquad +
\f{1}{2} {\tr}_{\cH} \Big(\big\{\ol{(A - z I_{\cH})^{1/2}(A_0 - z I_{\cH})^{-1}
(A - z I_{\cH})^{1/2}}\big\}^{-1}    \no \\
& \qquad \qquad \qquad \, \times
(A - z I_{\cH})^{1/2} (A_0 - z I_{\cH})^{-1}    \no \\
& \qquad \qquad \qquad \, \times \ol{\big[(A - z I_{\cH})^{-1} - (A_0 - z I_{\cH})^{-1}\big] (A - z I_{\cH})^{1/2}}\Big) \no \\
& \quad = \f{1}{2} {\tr}_{\cH} \Big(\big\{\ol{(A - z I_{\cH})^{1/2}(A_0 - z I_{\cH})^{-1}
(A - z I_{\cH})^{1/2}}\big\}^{-1}    \no \\
& \qquad \qquad \quad \, \times
(A - z I_{\cH})^{1/2}\big[(A - z I_{\cH})^{-1} - (A_0 - z I_{\cH})^{-1}\big]   \no \\
& \qquad \qquad \quad \, \times (A - z I_{\cH})^{-1/2}
\big\{\ol{(A - z I_{\cH})^{1/2} (A_0 - z I_{\cH})^{-1} (A - z I_{\cH})^{1/2}}\big\}\Big)
\no \\
& \qquad +
\f{1}{2} {\tr}_{\cH} \Big(\big\{\ol{(A - z I_{\cH})^{1/2}(A_0 - z I_{\cH})^{-1}
(A - z I_{\cH})^{1/2}}\big\}^{-1}    \no \\
& \qquad \qquad \qquad \, \times
\big\{\ol{(A - z I_{\cH})^{1/2} (A_0 - z I_{\cH})^{-1} (A - z I_{\cH})^{1/2}}\big\}
(A - z I_{\cH})^{-1/2}     \no \\
& \qquad \qquad \qquad \, \times
\big[\ol{(A - z I_{\cH})^{-1} - (A_0 - z I_{\cH})^{-1}\big]
(A - z I_{\cH})^{1/2}}\Big) \no \\
& \quad = \f{1}{2} {\tr}_{\cH} \Big(
(A - z I_{\cH})^{1/2}\big[(A - z I_{\cH})^{-1} - (A_0 - z I_{\cH})^{-1}\big]
(A - z I_{\cH})^{-1/2}\Big)
\no \\
& \qquad +
\f{1}{2} {\tr}_{\cH} \Big((A - z I_{\cH})^{-1/2}
\ol{\big[(A - z I_{\cH})^{-1} - (A_0 - z I_{\cH})^{-1}\big]
(A - z I_{\cH})^{1/2}}\Big) \no \\
& \quad = \f{1}{2} {\tr}_{\cH} \big((A - z I_{\cH})^{-1} - (A_0 - z I_{\cH})^{-1}\big)
+ \f{1}{2} {\tr}_{\cH} \big((A - z I_{\cH})^{-1} - (A_0 - z I_{\cH})^{-1}\big) \no \\
& \quad = {\tr}_{\cH} \big((A - z I_{\cH})^{-1} - (A_0 - z I_{\cH})^{-1}\big),
\quad z \in \bbC \backslash \big(\ol{- t_0 + S_{\omega_0}}\big).  \lb{7.23}
\end{align}
Here we repeatedly used cyclicity of the trace.
\end{proof}

\begin{remark} \lb{r7.9} 
$(i)$ Extensions of the standard perturbation determinant
\begin{equation}
{\det}_{\cH}\big((A - z I_{\cH})(A_0 - z I_{\cH})^{-1}\big)
= {\det}_{\cH}\big(I_{\cH} + (A - A_0)(A_0 - z I_{\cH})^{-1}\big)
\end{equation}
to certain symmetrized (sometimes called, modified) versions involving
factorizations of $A-A_0$ have been considered in \cite{KY81} and
\cite[Sect.\ 8.1.4]{Ya92}. However, Theorem \ref{t7.8} appears to be of a
more general nature and of independent interest. \\
$(ii)$ We emphasize the general nature of the hypotheses on $A, A_0$ 
in Theorem \ref{t7.8}. In particular, it covers the frequently encountered special case of self-adjoint operators $A, A_0$ with $A_0$ bounded 
from below and $A$ a quadratic form perturbation of $A_0$ with relative bound strictly less than one, in addition to the trace class requirement \eqref{7.20}. In this case one has  
$\dom\big(|A_0|^{1/2}\big) = \dom\big(|A|^{1/2}\big)$. 
Actually, Theorem \ref{t7.8} permits the more general situation where the latter equality of form domains is replaced by 
$\dom\big(|A_0|^{1/2}\big) \subseteq \dom\big(|A|^{1/2}\big)$. The latter  fact will have to be used in our application to one-dimensional 
Schr\"odinger operators on a compact interval in Section \ref{s8} in the case where the separated boundary conditions involve a Dirichlet 
boundary condition at one or both interval endpoints. \\
$(iii)$ Going beyond item $(VIII)$, we also note that Theorem \ref{t7.8} applies when $A, A_0$ are (Dunford) spectral operators of scalar type 
\cite[Ch.\ XVIII]{DS88} (in the sense that their resolvent is similar to the resolvent of a self-adjoint operator) with real spectrum bounded from 
below.
\end{remark}

\section{Boundary Data Maps and Their Basic Properties}   \label{s2}

This section is devoted to a brief review of boundary data maps as recently
introduced in \cite{CGM10}. The results taken from \cite{CGM10} are presented
without proof (for detailed proofs and for an extensive bibliography we refer to
\cite{CGM10}). We will also present a few new results of boundary data maps in this section (and then of course supply proofs).

Taking $R>0$, and fixing $\te_0, \te_R\in S_{2 \pi}$, with $S_{2 \pi}$ the strip
\begin{equation}
S_{2 \pi}=\{z\in\bbC\,|\, 0\leq \Re(z) < 2 \pi\},
\end{equation}
we introduce the linear map $\gate$, the trace map associated with the boundary
$\{0,R\}$ of $(0,R)$ and the parameters $\te_0, \te_R$, by
\begin{equation}\label{2.2}
\gate \colon \begin{cases}
C^1({[0,R]}) \rightarrow \bbC^2, \\
u \mapsto \begin{bmatrix} \cos(\te_0)u(0) + \sin(\te_0)u'(0)\\
\cos(\te_R)u(R) - \sin(\te_R)u'(R) \end{bmatrix}, \end{cases}
\quad  \te_0, \te_R\in S_{2 \pi},
\end{equation}
where ``prime'' denotes $d/dx$. We note, in particular, that the Dirichlet trace
$\gamma_D$, and the Neumnann trace $\gamma_N$ (in connection with the outward pointing unit normal vector at $\partial (0,R) = \{0, R\}$), are given by
\begin{equation}
\gamma_D = \gamma_{0,0} = - \gamma_{\pi,\pi}, \quad
\gamma_N = \gamma_{3\pi/2,3\pi/2} = - \gamma_{\pi/2,\pi/2}.   \lb{2.2AA}
\end{equation}

Next, assuming
\begin{equation}
V\in L^1((0,R); dx),    \lb{2.2aa}
\end{equation}
we introduce the following family of densely defined closed linear operators
$\Hte$ in $L^2((0,R); dx)$,
\begin{align}
& \Hte f= -f'' + Vf,  \quad \te_0, \te_R\in S_{2 \pi},     \no   \\
& f\in\dom(\Hte)=\big\{ g \in L^2((0,R); dx)\,\big|\, g, g'\in AC({[0,R]}); \, \gate (g)=0;
\label{2.2a} \\
& \hspace*{6.75cm} (-g''+Vg)\in L^2((0,R); dx)\big\}.   \no
\end{align}
Here $AC([0,R])$ denotes the set of absolutely continuous functions on $[0,R]$.
We remark that $V$ is not assumed to be real-valued in this section.
It is well-known that the spectrum of $\Hte$, $\sigma (\Hte)$ is purely discrete,
\begin{equation}
\sigma (\Hte) = \sigma_{\rm d}(\Hte), \quad \te_0, \te_R\in S_{2 \pi}.
\end{equation}
In addition, the resolvent of $\Hte$ is a Hilbert--Schmidt operator in $L^2((0,R); dx)$
and the eigenvalues $E_{\te_0,\te_R,n}$ of
$\Hte$, in the case of the separated boundary conditions at hand, are of the form
\begin{equation}
E_{\te_0,\te_R,n} \underset{n\to\infty}{=}  [(n\pi/R) + (a_n/n)]^2 \,
\text{ with $a_n\in \ell^\infty(\bbN)$},
\end{equation}
as shown in \cite[Lemma 1.3.3]{Ma86}. Moreover, $\Hte$ is known
to be $m$-sectorial (cf.\ \cite[Sect.\ III.6]{EE89}, \cite[Sect.\ VI.2.4]{Ka80}).

One notices that
\begin{equation}
\gamma_{(\theta_0 + \pi) \, \text{\rm mod} (2\pi), (\theta_R + \pi) \, \text{\rm mod} (2\pi)}
= - \gate,  \quad  \te_0, \te_R\in S_{2 \pi},     \lb{2.2A}
\end{equation}
and, on the other hand,
\begin{equation}
H_{(\theta_0 + \pi) \, \text{\rm mod} (2\pi), (\theta_R + \pi) \, \text{\rm mod} (2\pi)} = \Hte,
\quad  \te_0, \te_R\in S_{2 \pi},     \lb{2.2B}
\end{equation}
hence it suffices to consider $\te_0, \te_R\in S_{\pi}=\{z\in\bbC\,|\, 0\leq \Re(z) < \pi\}$ rather than $\te_0, \te_R\in S_{2 \pi}$ in connection with $\Hte$, but for simplicity of notation we will keep using the strip $S_{2 \pi}$ throughout
this manuscript.

The adjoint of $\Hte$ is given by
\begin{align}
& (\Hte)^* f= -f'' + {\ol V} f,  \quad \te_0, \te_R\in S_{2 \pi},     \no   \\
& f\in\dom\big((\Hte)^*\big)=\big\{ g \in L^2((0,R); dx)\,\big|\, g, g'\in AC({[0,R]});
\, \gamma_{\ol{\theta_0},\ol{\theta_R}} (g)=0;  \no \\
& \hspace*{6.2cm}(-g''+{\ol V}g)\in L^2((0,R); dx)\big\}.
\end{align}

Having described the operator $\Hte$ is some detail, still assuming \eqref{2.2aa},
we now briefly recall the corresponding closed, sectorial, and densely defined
sequilinear form, denoted by $Q_{\Hte}$, associated with $\Hte$ (cf.\
\cite[p.\ 312, 321, 327--328]{Ka80}):
\begin{align}
& Q_{\Hte}(f,g) = \int_0^R dx \big[\ol{f'(x)} g'(x) + V(x) \ol{f(x)} g(x)\big]   \no \\
& \hspace*{2.35cm} - \cot(\te_0) \ol{f(0)} g(0) -  \cot(\te_R) \ol{f(R)} g(R), \lb{2.2f} \\
& f, g \in \dom(Q_{\Hte}) = H^1((0,R))  \no \\
& \quad = \big\{h\in L^2((0,R); dx) \,|\,
h \in AC ([0,R]); \, h' \in L^2((0,R); dx)\big\},   \no \\
& \hspace*{6.5cm}  \te_0, \te_R \in S_{2 \pi}\backslash\{0,\pi\},   \no \\
& Q_{H_{0,\te_R}}(f,g) = \int_0^R dx \big[\ol{f'(x)} g'(x) + V(x) \ol{f(x)} g(x)\big]
-  \cot(\te_R) \ol{f(R)} g(R),  \lb{2.2g} \\
&  f, g \in \dom(Q_{H_{0,\te_R}})   \no \\
& \quad = \big\{h\in L^2((0,R); dx) \,|\,
h \in AC ([0,R]); \, h(0) =0; \, h' \in L^2((0,R); dx)\big\},   \no \\
& \hspace*{8.5cm}  \te_R \in S_{2 \pi}\backslash\{0,\pi\},   \no \\
& Q_{H_{\te_0,0}}(f,g) = \int_0^R dx \big[\ol{f'(x)} g'(x) + V(x) \ol{f(x)} g(x)\big]
-  \cot(\te_0) \ol{f(0)} g(0),  \lb{2.2h} \\
&  f, g \in \dom(Q_{H_{\te_0,0}})   \no \\
& \quad = \big\{h\in L^2((0,R); dx) \,|\,
h \in AC ([0,R]); \, h(R) =0; \, h' \in L^2((0,R); dx)\big\},   \no \\
& \hspace*{8.67cm}  \te_0 \in S_{2 \pi}\backslash\{0,\pi\},   \no \\
& Q_{H_{0,0}}(f,g) = \int_0^R dx \big[\ol{f'(x)} g'(x) + V(x) \ol{f(x)} g(x)\big],  \lb{2.2i} \\
& f, g \in \dom(Q_{H_{0,0}}) = \dom\big(|H_{0,0}|^{1/2}\big) = H^1_0((0,R))  \no \\
& \quad = \big\{h\in L^2((0,R); dx) \,|\,
h \in AC([0,R]); \, h(0)=0, \, h(R) =0;   \no \\
& \hspace*{6.65cm}   h' \in L^2((0,R); dx)\big\}.   \no
\end{align}

Next, we recall the following elementary, yet fundamental, fact:

\begin{lemma} \lb{l2.2}
Suppose that $V\in L^1((0,R); dx)$, fix $\te_0, \te_R\in S_{2 \pi}$, and assume
that $z\in\CR$. Then the boundary value problem given by
\begin{align}
-&u'' + Vu=zu,\quad u, u'\in AC([0,R]),   \label{2.3}\\
&\gate(u)=\begin{bmatrix}c_0\\ c_R \end{bmatrix}\in\bbC^2,    \label{2.4}
\end{align}
has a unique solution $u(z,\cdot)=u(z,\cdot\,;(\te_0,c_0),(\te_R,c_R))$ for each $c_0, c_R\in\bbC$.
\end{lemma}

Assuming $z\in\rho(\Hte)$, a basis for the solutions of
\eqref{2.3} is given by
\begin{equation}\label{2.5}
\begin{split}
u_{-,\te_0}(z,\cdot)&=u(z,\cdot \, ;(\te_0,0),(0,1)),    \\
u_{+,\te_R}(z,\cdot)&=u(z,\cdot \, ;(0,1),(\te_R,0)).
\end{split}
\end{equation}
Explicitly, one then has
\begin{align}
u_{-,\te_0}(z,R)&= 1, \quad
\cos(\te_0)u_{-,\te_0}(z,0) + \sin(\te_0)u_{-,\te_0}'(z,0) =0,    \label{2.5a}\\
u_{+,\te_R}(z,0)&= 1, \quad
\cos(\te_R)u_{+,\te_R}(z,R) - \sin(\te_R)u_{+,\te_R}'(z,R) =0.     \label{2.5b}
\end{align}
Recalling the Wronskian of two functions $f$ and $g$,
\begin{equation}
W(f,g)(x) = f(x)g'(x) - f'(x)g(x), \quad f, g \in C^1([0,R]),   \lb{2.5bA}
\end{equation}
one then computes
\begin{align}
W(u_{+,\te_R}(z,\cdot), u_{-,\te_0}(z,\cdot))
& = u'_{-,\te_0}(z,0)- u'_{+,\te_R}(z,0) u_{-,\te_0}(z,0)     \lb{2.5c} \\
& = u_{+,\te_R}(z,R) u'_{-,\te_0}(z,R)- u'_{+,\te_R}(z,R).   \lb{2.5d}
\end{align}

To each boundary value problem \eqref{2.3}, \eqref{2.4}, we now associate a family of \emph{general boundary data maps}, $\Lates (z) : \bbC^2 \rightarrow \bbC^2$, for
$\te_0, \te_R, \te_0^{\prime},\te_R^{\prime}\in S_{2 \pi}$,  where
\begin{align}\label{2.6}
\begin{split}
\Lates (z) \begin{bmatrix}c_0\\ c_R \end{bmatrix} &=
\Lates (z) \big(\gate(u(z,\cdot\ ;(\te_0,c_0),(\te_R,c_R)))\big)   \\
&= \gates(u(z,\cdot\ ;(\te_0,c_0),(\te_R,c_R))).
\end{split}
\end{align}
With $u(z,\cdot)=u(z,\cdot\ ;(\te_0,c_0),(\te_R,c_R))$, then $\Lates (z) $ can be represented as a $2\times 2$ complex matrix, where
\begin{align}\label{2.7}
\Lates (z) \begin{bmatrix}c_0\\ c_R \end{bmatrix} &=
\Lates (z) \begin{bmatrix} \cos(\te_0)u(z,0) + \sin(\te_0)u'(z,0)\\[1mm]
\cos(\te_R)u(z,R) - \sin(\te_R)u'(z,R) \end{bmatrix}  \no \\
&=
\begin{bmatrix} \cos(\te_0^{\prime})u(z,0) + \sin(\te_0^{\prime})u'(z,0)\\[1mm]
\cos(\te_R^{\prime})u(z,R) - \sin(\te_R^{\prime})u'(z,R) \end{bmatrix}.
\end{align}

One can show that $\Lates$ is well-defined for
$z\in\rho(\Hte)$, that is, it is invariant with respect to a change of basis of solutions of \eqref{2.3} (cf.\ \cite[Theorem\ 2.3]{CGM10}). Moreover, one has the following facts:
Let $\te_0,\te_R, \te_0^{\prime},\te_R^{\prime}, \te_0^{\prime\prime}, \te_R^{\prime\prime}\in S_{2 \pi}$. Then, with $I_2$ denoting the identity matrix in $\bbC^2$,
\begin{align}
& \La_{\te_0,\te_R}^{\te_0,\te_R} (z) = I_2,   \quad
z\in\rho(\Hte),    \lb{2.7a}  \\
& \La_{\te_0^{\prime},\te_R^{\prime}}^{\te_0^{\prime\prime},\te_R^{\prime\prime}} (z)
\Lates (z) =\La_{\te_0,\te_R}^{\te_0^{\prime\prime},\te_R^{\prime\prime}} (z) , \quad
z\in\rho(\Hte)\cap\rho(H_{\theta'_0,\theta'_R}),
\lb{2.7b} \\
& \La_{\te_0^{\prime},\te_R^{\prime}}^{\te_0,\te_R} (z) = \Big[\Lates (z)\Big]^{-1},
\quad
z\in\rho(\Hte)\cap\rho(H_{\theta'_0,\theta'_R}).
\lb{2.7c}
\end{align}

\begin{remark} \lb{r2.5}
Even though $\Lates$ is invariant with respect to a change of basis for the solutions of \eqref{2.3}, the representation of $\Lates$ with respect to a specific basis can be simplified considerably with an appropriate choice of basis. For example, by choosing the basis given in \eqref{2.5}, and by letting
$\psi_1(z,\cdot)=u_{+,\te_R}(z,\cdot)=u(z,\cdot \, ;(0,1),(\te_R,0))$, and
$\psi_2(z,\cdot)=u_{-,\te_0}(z,\cdot)=u(z,\cdot \, ;(\te_0,0),(0,1))$, one obtains using
this basis,
\begin{align}
& \Lates (z) = \Big[\Big(\Lates (z)\Big)_{j,k}\Big]_{1 \leq j,k \leq 2},
\quad z\in\rho(\Hte),    \label{2.20} \\[1mm]
\begin{split}
& \Big(\Lates (z)\Big)_{1,1} = \frac{\cos(\te_0^{\prime})
+ \sin(\te_0^{\prime})u_{+,\te_R}'(z,0)}{\cos(\te_0) + \sin(\te_0)u_{+,\te_R}'(z,0)},  \\
& \Big(\Lates (z)\Big)_{1,2} = \frac{\cos(\te_0^{\prime})u_{-,\te_0}(z,0) + \sin(\te_0^{\prime})u_{-,\te_0}'(z,0)}{\cos(\te_R) - \sin(\te_R)u_{-,\te_0}'(z,R)},    \\
& \Big(\Lates (z)\Big)_{2,1} = \frac{\cos(\te_R^{\prime})u_{+,\te_R}(z,R)
- \sin(\te_R^{\prime})u_{+,\te_R}'(z,R)}{\cos(\te_0) + \sin(\te_0)u_{+,\te_R}'(z,0)},   \\
& \Big(\Lates (z)\Big)_{2,2} = \frac{\cos(\te_R^{\prime})
- \sin(\te_R^{\prime})u_{-,\te_0}'(z,R)}{\cos(\te_R) - \sin(\te_R)u_{-,\te_0}'(z,R)}.
\lb{2.20a}
\end{split}
\end{align}
In particular, by \eqref{2.5a} and \eqref{2.5b},
\begin{equation}
\Big(\Lambda_{\te_0,\te_R}^{\te_0,\te_R'} (z)\Big)_{1,2} = 0, \quad
\Big(\Lambda_{\te_0,\te_R}^{\te_0',\te_R} (z)\Big)_{2,1} = 0.      \lb{2.30}
\end{equation}
\end{remark}

\begin{remark} \lb{r2.6}
We note that $\Lazzqq(z)$ represents the \emph{Dirichlet-to-Neumann map},
$\Lambda_{D,N}(z)$, for the boundary value problem \eqref{2.3}, \eqref{2.4}; that is,  when
$\te_0=\te_R=0$, $\te_0^{\prime}=\te_R^{\prime}=\pi/2$, then \eqref{2.7} becomes
\begin{equation}
\Lambda_{D,N}(z) \begin{bmatrix}u(z,0)\\u(z,R)\end{bmatrix}=
\Lazzqq( z) \begin{bmatrix}u(z,0)\\u(z,R)\end{bmatrix}=
\lbrac\begin{array}{r}u'(z,0)\\-u'(z,R)\end{array}\rbrac, \quad z\in\rho(\Hte),
\end{equation}
with $u(z,\cdot)=u(z,\cdot \, ;(0,c_0),(0,c_R))$, $u(z,0)=c_0$, $u(z,R)=c_R$. The Dirichlet-to-Neumann map in the case $V=0$ has recently been considered in
\cite[Example 5.1]{Po08}. The Neumann-to-Dirichlet map $\Lambda_{N,D}(z)
= \Lambda_{\pi/2,\pi/2}^{\pi,\pi} (z) = - [\Lambda_{D,N}(z)]^{-1}$
in the case $V=0$
has earlier been computed in \cite[Example\ 4.1]{DM92}. We also refer to
\cite{BMN08}, \cite{BGP08}, \cite{DM91} in the intimately related
context of $Q$ and $M$-functions.
\end{remark}

We continue with an elementary result needed in the proof of
Lemma \ref{l2.8}, but first we introduce a convenient basis of solutions
associated with the Schr\"odinger equation \eqref{2.3}: Fix $z\in\bbC$ and
let $\theta(z,\cdot), \theta^{\prime}(z,\cdot), \phi(z,\cdot),
\phi^{\prime}(z,\cdot)\in AC([0,R])$, and such that $\theta (z,\cdot)$ and
$\phi (z,\cdot)$ are solutions of
\begin{equation}
-u'' +V u = z u,    \lb{2.32a}
\end{equation}
uniquely determined by their initial values at $x=0$,
\begin{equation}
\theta (z,0)=\phi^{\prime}(z,0)=1, \quad \theta^{\prime}(z,0)=\phi (z,0)=0.  \lb{2.2b}
\end{equation}
In particular, $\theta (z,\cdot)$  and $\phi (z,\cdot)$ are entire with respect to $z$.
Introducing
\begin{equation}
\psi(z,\cdot) = A \theta (z,\cdot) + B \phi (z,\cdot), \quad A, B \in \bbC, \lb{2.2bb}
\end{equation}
it follows that
\begin{equation}
\gate(\psi)
= \begin{bmatrix} \cos(\te_0) \psi(z,0) + \sin(\te_0) \psi'(z,0) \\[1mm]
\cos(\te_R) \psi(z,R) - \sin(\te_R) \psi'(z,R) \end{bmatrix} = 0 \in \bbC^2  \lb{2.2c}
\end{equation}
is equivalent to
\begin{align}
0 &= \begin{bmatrix}  \cos(\te_0) &  \sin(\te_0) \\[1mm]
\cos(\te_R) \theta (z,R) - \sin(\te_R) \theta' (z,R) &
\cos(\te_R) \phi(z,R) - \sin(\te_R) \phi'(z,R) \end{bmatrix}
\begin{bmatrix} A \\[1mm]  B \end{bmatrix}   \no \\
& = \cU(z, R, \te_0, \te_R) \begin{bmatrix} A \\ B \end{bmatrix}.    \lb{2.2ca}
\end{align}
Consequently, introducing the determinant $\Delta$ defined by
\begin{equation}
\Delta(z,R,\te_0,\te_R) =\det\big(\cU(z, R, \te_0, \te_R)\big),    \lb{2.2det}
\end{equation}
one concludes that
\begin{align}
\text{$z_0$ is an eigenvalue of $\Hte$}  \Longleftrightarrow
\text{$z_0$ is a zero of the determinant $\Delta(\cdot,R,\te_0,\te_R)$.}
\end{align}
Moreover, $\Delta$ is an entire function with respect to $z$, and an explicit computation reveals that
\begin{align}
\begin{split}
\Delta(z,R,\te_0,\te_R)&=\cos(\te_0)\cos(\te_R)\phi (z,R)
- \cos(\te_0)\sin(\te_R) \phi^{\prime} (z,R)  \\
& \quad - \sin(\te_0)\cos(\te_R) \theta (z,R) + \sin(\te_0)\sin(\te_R) \theta^{\prime} (z,R).   \lb{2.2d}
\end{split}
\end{align}
In addition, we point out that the function $\De$ is closely related to the usual Wronskian of two solution $u_{\pm,\te_0,\te_R}$ of \eqref{2.32a} satisfying the boundary conditions
\begin{align}
\cos(\te_0)u_{+,\te_0,\te_R}(z,0) + \sin(\te_0)u_{+,\te_0,\te_R}'(z,0)&=1, \\
\cos(\te_R)u_{+,\te_0,\te_R}(z,R) - \sin(\te_R)u_{+,\te_0,\te_R}'(z,R)&=0, \\
\cos(\te_0)u_{-,\te_0,\te_R}(z,0) + \sin(\te_0)u_{-,\te_0,\te_R}'(z,0)&=0, \\
\cos(\te_R)u_{-,\te_0,\te_R}(z,R) - \sin(\te_R)u_{-,\te_0,\te_R}'(z,R)&=1.
\end{align}
In vector form, these boundary conditions correspond to
\begin{align}
\begin{bmatrix} \ga_{\te_0,\te_R}(u_{+,\te_0,\te_R}) & \ga_{\te_0,\te_R}(u_{-,\te_0,\te_R}) \end{bmatrix} = I_2. \lb{2.2da}
\end{align}
Since $\ga_{\te_0,\te_R}$ is a linear map, it follows from \eqref{2.2bb} that
\begin{align}
\ga_{\te_0,\te_R}(\psi) =
\begin{bmatrix} \ga_{\te_0,\te_R}(\te) & \ga_{\te_0,\te_R}(\phi) \end{bmatrix}
\begin{bmatrix} A \\ B \end{bmatrix},
\end{align}
hence by \eqref{2.2ca}
\begin{align}
\cU(z, R, \te_0, \te_R) = \begin{bmatrix} \ga_{\te_0,\te_R}(\te) & \ga_{\te_0,\te_R}(\phi) \end{bmatrix}. \lb{2.2db}
\end{align}
Using \eqref{2.2da} and the linearity of $\ga_{\te_0,\te_R}$ once again one concludes that
\begin{align}
\begin{bmatrix} u_{+,\te_0,\te_R}(z,x) & u_{-,\te_0,\te_R}(z,x) \end{bmatrix} =
\begin{bmatrix} \te(z,x) & \phi(z,x) \end{bmatrix}
\begin{bmatrix} \ga_{\te_0,\te_R}(\te) & \ga_{\te_0,\te_R}(\phi) \end{bmatrix}^{-1} \lb{2.2dc}
\end{align}
since both sides solve \eqref{2.32a} and satisfy the same boundary condition. Thus, \eqref{2.2db} and \eqref{2.2dc} yield
\begin{align}
W(u_{+,\te_0,\te_R}(z,\cdot), u_{-,\te_0,\te_R}(z,\cdot)) &= W(\te(z,\cdot),\phi(z,\cdot))\det(\cU(z, R, \te_0, \te_R)^{-1}) \no \\
&= \Delta(z,R,\te_0,\te_R)^{-1}, \lb{2.2dd}
\end{align}
where $W(\cdot,\cdot)$ denotes the Wronskian of two functions as introduced in \eqref{2.5bA}.

\begin{lemma} \lb{l2.8}
Assume that $\te_0, \te_R, \te_0', \te_R' \in S_{2 \pi}$, and let $\Hte$ and
$H_{\theta_0',\theta_R'}$ be defined as in \eqref{2.2a}. Then, with
$\Delta(\cdot,R,\te_0,\te_R)$ introduced in \eqref{2.2det},
\begin{align}
{\det}_{\bbC^2} \Big(\Lates (z)\Big) =
\f{\Delta(z,R,\te_0',\te_R')}{\Delta(z,R,\te_0,\te_R)}, \quad
z\in\rho(\Hte).      \lb{2.43}
\end{align}
\end{lemma}
\begin{proof}
We recall the formula
\begin{equation}\label{2.44}
\Lates (z) = \begin{bmatrix}\ga_{\te_0^{\prime},\te_R^{\prime}}(\te(z,\cdot)) &
\ga_{\te_0^{\prime},\te_R^{\prime}}(\phi(z,\cdot))\end{bmatrix}
\begin{bmatrix}\ga_{\te_0,\te_R}(\te(z,\cdot)) &\ga_{\te_0,\te_R}(\phi(z,\cdot)
\end{bmatrix}^{-1},
\end{equation}
established in the proof of Theorem\ 2.3 in \cite{CGM10}. Since
\begin{align}
& \begin{bmatrix}\ga_{\te_0,\te_R}(\te(z,\cdot)) &\ga_{\te_0,\te_R}(\phi(z,\cdot)
\end{bmatrix}   \lb{2.45} \\
& \quad = \begin{pmatrix} \cos(\te_0) & \sin(\te_0) \\
\cos(\te_R) \te(z,R) - \sin(\te_R) \te'(z,R) & \cos(\te_R) \phi(z,R)
- \sin(\te_R)\phi'(z,R) \end{pmatrix},    \no
\end{align}
one verifies
\begin{equation}
{\det}_{\bbC^2} \big(\begin{bmatrix}\ga_{\te_0,\te_R}(\te(z,\cdot))
&\ga_{\te_0,\te_R}(\phi(z,\cdot) \end{bmatrix}\big) = \Delta(z,R,\te_0,\te_R),
\quad z\in\bbC,
\end{equation}
and hence \eqref{2.43}.
\end{proof}

The following asymptotic expansion results will be used in the proof of
Theorem \ref{t4.3}:

\begin{lemma} \lb{l2.9}
Assume that $\te_0, \te_R \in S_{2 \pi}$,
$\te_0', \te_R' \in S_{2 \pi}\backslash\{0,\pi\}$, and let $\Hte$ and
$H_{\theta_0',\theta_R'}$ be defined as in \eqref{2.2a}. Then,
\begin{align}
& {\det}_{\bbC^2}\Big(\Lates (z)\Big)   \no \\
& \quad \underset{\substack{|z| \to \infty\\\Im(z^{1/2})>0}}{=}
\begin{cases}
\f{\sin(\te_0') \sin(\te_R')}{\sin(\te_0) \sin(\te_R)} + \Oh(|z|^{-1/2}),
& \te_0, \te_R \in S_{2 \pi}\backslash \{0,\pi\}, \\[1mm]
- \f{\sin(\te_0') \sin(\te_R')}{\sin(\te_R)} |z|^{1/2} + \Oh(1),
& \te_0 = 0, \, \te_R \in S_{2 \pi}\backslash \{0,\pi\}, \\[1mm]
- \f{\sin(\te_0') \sin(\te_R')}{\sin(\te_0)} |z|^{1/2} + \Oh(1),
& \te_0 \in S_{2 \pi}\backslash \{0,\pi\}, \, \te_R = 0, \\[1mm]
\sin(\te_0') \sin(\te_R') |z| + \Oh(|z|^{1/2}), & \te_0 = \te_R = 0.
\end{cases}       \lb{2.46}
\end{align}
\end{lemma}
\begin{proof}
The standard Volterra integral equations
\begin{align}
\theta (z,x) &= \cos(z^{1/2}x) + \int_0^x dx' \, \f{\sin(z^{1/2}(x-x'))}{z^{1/2}} V(x')
\theta (z,x'),  \\
\phi (z,x) &= \f{\sin(z^{1/2}x)}{z^{1/2}} + \int_0^x dx' \, \f{\sin(z^{1/2}(x-x'))}{z^{1/2}} V(x')
\phi (z,x'),   \\
& \hspace*{3.5cm}  z\in\bbC, \; \Im(z^{1/2}) \geq 0, \; x\in [0,R],   \no
\end{align}
readily imply that
\begin{align}
\begin{split}
\theta (z,x) & \underset{|z|\to\infty}{=} \cos(z^{1/2}x)
+ \Oh\Big(|z|^{-1/2}e^{\Im(z^{1/2})x}\Big), \\
\theta^{\prime}(z,x) & \underset{|z|\to\infty}{=} - z^{1/2} \sin(z^{1/2}x)
+ \Oh\Big(e^{\Im(z^{1/2})x}\Big), \\
\phi (z,x) & \underset{|z|\to\infty}{=} \f{\sin(z^{1/2}x)}{z^{1/2}}
+ \Oh\Big(|z|^{-1}e^{\Im(z^{1/2})x}\Big),    \lb{2.2e}  \\
\phi^{\prime}(z,x) & \underset{|z|\to\infty}{=} \cos(z^{1/2}x)
+ \Oh\Big(|z|^{-1/2}e^{\Im(z^{1/2})x}\Big).
\end{split}
\end{align}
An insertion of \eqref{2.2e} into \eqref{2.2d} then yields
\begin{align}
\Delta(z,R, \te_0,\te_R) \underset{\substack{|z| \to \infty\\ \Im(z^{1/2})>0}}{=}
\begin{cases}
2^{-1} \sin(\te_0) \sin(\te_R) |z|^{1/2} e^{\Im(z^{1/2})} + \Oh\Big(e^{\Im(z^{1/2}R)}\Big), \\
\hspace*{4.55cm} \te_0, \te_R \in S_{2 \pi}\backslash \{0,\pi\}, \\[1mm]
- 2^{-1} \sin(\te_R) e^{\Im(z^{1/2})} + \Oh\Big( |z|^{-1/2} e^{\Im(z^{1/2}R)}\Big), \\
\hspace*{3.25cm} \te_0 = 0, \, \te_R \in S_{2 \pi}\backslash \{0,\pi\}, \\[1mm]
- 2^{-1} \sin(\te_0) e^{\Im(z^{1/2})} + \Oh\Big( |z|^{-1/2} e^{\Im(z^{1/2}R)}\Big), \\
\hspace*{3.2cm} \te_0 \in S_{2 \pi}\backslash \{0,\pi\}, \, \te_R = 0, \\[1mm]
2^{-1} |z|^{-1/2} e^{\Im(z^{1/2})} + \Oh\Big( |z|^{-1} e^{\Im(z^{1/2})R}\Big), \\
\hspace*{4.35cm} \te_0 = \te_R = 0.
\end{cases}        \lb{2.48}
\end{align}
Finally, combining \eqref{2.43} and \eqref{2.48} proves \eqref{2.46}.
\end{proof}

Next, we recall an explicit formula for $\Lates(z)$ in terms of the
resolvent $(\Hte- z I \big)^{-1}$ of $\Hte$ and the boundary traces
$\gamma_{\theta'_0,\theta'_R}$. We start with the Green's function for the
operator $\Hte$ in \eqref{2.2a},
\begin{align}
& \Gte(z,x,x') = (\Hte -z I)^{-1}(x,x')   \no \\
& \quad = \frac{1}{W(u_{+,\te_R}(z,\cdot), u_{-,\te_0}(z,\cdot))} \begin{cases}
u_{-,\te_0}(z,x')u_{+,\te_R}(z,x), & 0\le x'\le x,\\[1mm]
u_{-,\te_0}(z,x)u_{+,\te_R}(z,x'), & 0\le x\le x',
\end{cases}   \label{3.32} \\[1mm]
&\hspace*{6.35cm}  z\in\rho(\Hte), \; x, x' \in {[0,R]}.  \no
\end{align}
Here $u_{+,\te_R}(z,\cdot), u_{-,\te_0}(z,\cdot)$ is a basis for solutions of \eqref{2.3} as described in \eqref{2.5} and we denote by $I = I_{L^2((0,R); dx)}$ the identity operator
in $L^2((0,R); dx)$. Thus, one obtains
\begin{align}
\begin{split}
 \big((\Hte- z I)^{-1}g\big)(x) = \int_0^R \, dx'\, \Gte(z,x,x')g(x'),& \\
 g\in L^2((0,R); dx), \; z\in\rho(\Hte), \; x\in(0,R).&  \lb{3.33}
\end{split}
\end{align}

For future purposes we now introduce the following $2 \times 2$ matrix
\begin{equation}\label{2.14}
\Ste=\begin{bmatrix} \sin(\te_0)&0\\0&\sin(\te_R)\end{bmatrix}.
\end{equation}

\begin{theorem} \lb{t3.5}
Assume that $\te_0, \te_R, \te_0', \te_R' \in S_{2 \pi}$ and let $\Hte$ be defined
as in \eqref{2.2a}. Then
\begin{equation}
 \Lates (z) S_{\theta_0' -\theta_0,\theta_R'-\theta_R}
=\gamma_{\te_0',\te_R'}
\big[\gamma_{\ol{\theta_0'},\ol{\theta_R'}} ((\Hte)^* - {\ol z} I)^{-1}\big]^*,
\quad z\in\rho(\Hte).
\lb{3.33aa}
\end{equation}
\end{theorem}

The fact that $\Lates(z)$ and $\Lades(z)$ satisfy a linear fractional transformation is recalled next:

\begin{theorem}  \lb{t4.1}
Assume that $\te_0, \te_R, \te_0^{\prime},\te_R^{\prime}, \de_0,\de_R,\de_0^\prime,\de_R^\prime\in S_{2 \pi}$, $\de_0^\prime-\de_0\ne 0 \, \text{\rm mod} (\pi)$,
$\de_R^\prime-\de_R \ne 0 \, \text{\rm mod} (\pi)$, and that
$z\in\rho(H_{\te_0,\te_R})\cap\rho(H_{\de_0,\de_R})$. Then, with $\Ste$
defined as in \eqref{2.14},
\begin{align}
\begin{split}
\Lates(z) &= \big(S_{\de_0^\prime-\de_0,\de_R^\prime-\de_R}\big)^{-1}
\big[S_{\de_0^\prime-\te_0^\prime,\de_R^\prime-\te_R^\prime}+
S_{\te_0^\prime-\de_0,\te_R^\prime-\de_R}\Lades(z) \big]  \label{2.53} \\
& \quad \times
\big[S_{\de_0^\prime-\te_0,\de_R^\prime-\te_R}+
S_{\te_0-\de_0,\te_R-\de_R}\Lades(z) \big]^{-1}S_{\de_0^\prime-\de_0,\de_R^\prime-\de_R}.
\end{split}
\end{align}
If in addition, $\te_R^\prime-\te_R \ne 0 \, \text{\rm mod} (\pi)$,
$\de_0^\prime-\de_0\ne 0 \, \text{\rm mod} (\pi)$, then
\begin{align} \lb{2.54}
& \Lates(z) S_{\te_0'-\te_0,\te_R'-\te_R}   \no \\
& \quad = \Big[S_{\de_0^\prime-\te_0^\prime,\de_R^\prime-\te_R^\prime} +
\big(S_{\de_0^\prime-\de_0,\de_R^\prime-\de_R}\big)^{-1}
S_{\te_0^\prime-\de_0,\te_R^\prime-\de_R}\Lades(z)
S_{\de_0^\prime-\de_0,\de_R^\prime-\de_R}\Big]      \no\\
& \qquad \, \times
\Big[\big(S_{\te_0'-\te_0,\te_R'-\te_R}\big)^{-1}
S_{\de_0^\prime-\te_0,\de_R^\prime-\te_R} +
\big(S_{\te_0'-\te_0,\te_R'-\te_R}\big)^{-1}
\big(S_{\de_0^\prime-\de_0,\de_R^\prime-\de_R}\big)^{-1}   \no \\
& \hspace*{4.5cm} \times S_{\te_0-\de_0,\te_R-\de_R}\Lades(z)
S_{\de_0'-\de_0,\de_R'-\de_R}\Big]^{-1}.
\end{align}
\end{theorem}

We denote by $\bbC_+$  the open complex upper half-plane and abbreviate
$\Im(L)=(L - L^*)/(2i)$ for $L \in\bbC^{n\times n}$, $n\in\bbN$. In addition,
$d \|\Sigma\|_{\bbC^{2 \times 2}}$ will denote the total variation of the
$2\times 2$ matrix-valued measure $d \Sigma$ below in \eqref{5.2}.

The matrix $M(\cdot)$ is called an $n \times n$ matrix-valued Herglotz function
if it is analytic on $\bbC_+$ and $\Im(M(z)) \geq 0$ for all $z\in\bbC_+$.
Now we are in position to recall the fundamental Herglotz property of the matrix
$\Lates (\cdot) S_{\te_0'-\te_0,\te_R'-\te_R}$ in the case where $\Hte$ is self-adjoint:

\begin{theorem} \lb{t4.6}
Let $\te_0, \te_R, \te_0', \te_R' \in [0,2 \pi)$,
$\te_0^\prime-\te_0\ne 0 \, \text{\rm mod} (\pi)$,
$\te_R^\prime-\te_R \ne 0 \, \text{\rm mod} (\pi)$,
$z\in\rho(\Hte)$, and $\Hte$
be defined as in \eqref{2.2a}. In addition, suppose that $V$ is real-valued $($and
hence $\Hte$ is self-adjoint\,$)$. Then $\Lates (\cdot) S_{\te_0'-\te_0,\te_R'-\te_R}$
is a $2\times 2$ matrix-valued Herglotz function admitting the representation
\begin{align}
& \Lates (z) S_{\te_0'-\te_0,\te_R'-\te_R} = \Xi_{\te_0,\te_R}^{\te_0', \te_R'}
+ \ \int_{{\mathbb{R}}}
d\Sigma_{\te_0,\te_R}^{\te_0', \te_R'} (\lambda)\bigg(\frac{1}{\lambda -z}-\frac{\lambda}
{1+\lambda^2}\bigg),     \lb{5.1} \\
& \hspace*{8.25cm} z\in\rho(\Hte),   \no \\
& \Xi_{\te_0,\te_R}^{\te_0', \te_R'}
= \Big(\Xi_{\te_0,\te_R}^{\te_0', \te_R'}\Big)^* \in\bbC^{2\times 2},
\quad \int_{\bbR} \f{d\big\|\Sigma_{\te_0,\te_R}^{\te_0', \te_R'}
(\lambda)\big\|_{\bbC^{2 \times 2}}}{1+\lambda^2} <\infty,
\lb{5.2}
\end{align}
where
\begin{align}
\begin{split}
\Sigma_{\te_0,\te_R}^{\te_0', \te_R'} ((\lambda_1,\lambda_2]) =
\f{1}{\pi}\lim_{\delta\downarrow 0}\lim_{\varepsilon\downarrow
0}\int^{\lambda_2+\delta}_{\lambda_1+\delta}d\lambda \,
\Im\Big(\Lates (\lambda +i\varepsilon) S_{\te_0'-\te_0,\te_R'-\te_R}\Big),&  \\
 \lambda_1, \lambda_2
\in\bbR, \; \lambda_1<\lambda_2.&  \lb{5.3}
\end{split}
\end{align}
In addition,
\begin{equation}
\Im\Big(\Lates (z) S_{\te_0'-\te_0,\te_R'-\te_R}\Big) > 0, \quad z\in\bbC_+,    \lb{5.3a}
\end{equation}
and
\begin{equation}
\supp\Big(d\Sigma_{\te_0,\te_R}^{\te_0', \te_R'}\Big)
\subseteq \big(\sigma(\Hte) \cup \sigma(H_{\te_0',\te_R'})\big),    \lb{5.3b}
\end{equation}
in particular, $\Lates (\cdot) S_{\te_0'-\te_0,\te_R'-\te_R}$ is self-adjoint on
$\bbR \cap \rho(\Hte) \cap \rho(H_{\te_0',\te_R'})$.
\end{theorem}

We note that relation \eqref{5.3b} is a consequence of \eqref{2.43} and of the fact that
$\Lates (\cdot) S_{\te_0'-\te_0,\te_R'-\te_R}$ is a meromorphic Herglotz matrix.

\begin{remark} \lb{r4.7}
If $\te_0^\prime-\te_0 = 0 \, \text{\rm mod} (\pi)$ and
$\te_R^\prime-\te_R \neq 0 \, \text{\rm mod} (\pi)$
(resp., $\te_R^\prime-\te_R = 0 \, \text{\rm mod} (\pi)$ and
$\te_0^\prime-\te_0 \neq 0 \, \text{\rm mod} (\pi)$) then \eqref{2.30} shows that
$\Lambda_{\te_0,\te_R}^{\te_0,\te_R'} (\cdot) S_{0,\te_R'-\te_R}$
(resp., $\Lambda_{\te_0,\te_R}^{\te_0',\te_R} (\cdot) S_{\te_0'-\te_0,0}$) is a
diagonal matrix of the form
\begin{equation}
\Lambda_{\te_0,\te_R}^{\te_0,\te_R'} (\cdot) S_{0,\te_R'-\te_R} =
\begin{pmatrix} 0 & 0 \\ 0 & m_{\te_R}(\cdot) \end{pmatrix}
\;\;
\left(\text{resp., } \, \Lambda_{\te_0,\te_R}^{\te_0',\te_R} (\cdot) S_{\te_0'-\te_0,0} =
\begin{pmatrix} m_{\te_0}(\cdot) & 0 \\ 0 & 0  \end{pmatrix} \right),
\end{equation}
with $m_{\te_R}(\cdot)$ (resp., $m_{\te_0}(\cdot)$) a scalar Herglotz function.
\end{remark}

Finally, we briefly turn to a discussion of Krein-type resolvent formulas for the difference of resolvents of $H_{\theta_0',\theta_R'}$ and $\Hte$:

\begin{lemma} \lb{l6.1}
Assume that $\te_0, \te_R, \te_0', \te_R' \in S_{2 \pi}$, let $\Hte$ be defined as in
\eqref{2.2a}, and suppose that $z\in\rho(\Hte)$. Then,
assuming $f \in L^2((0,R); dx)$, and writing
\begin{equation}
\gamma_{\te_0^{\prime},\te_R^{\prime}}(\Hte - z I)^{-1} f
= \begin{bmatrix}
\big(\gamma_{\te_0^{\prime},\te_R^{\prime}}(\Hte - z I)^{-1}\big)_1f \\[2mm]
\big(\gamma_{\te_0^{\prime},\te_R^{\prime}}(\Hte - z I)^{-1}\big)_2 f
\end{bmatrix} \in \bbC^2,    \lb{6.1}
\end{equation}
one has
\begin{align}
\big(\gamma_{\te_0^{\prime},\te_R^{\prime}}(\Hte - z I)^{-1}\big)_1f
&= \f{\sin(\te_0'-\te_0)}{W(u_{+,\te_R}(z,\cdot), u_{-,\te_0}(z,\cdot))}
\big(\ol{u_{+,\te_R}(z,\cdot)}, f\big)_{L^2((0,R); dx)}   \no \\
& \quad \times \begin{cases} - \f{u_{-,\te_0}(z,0)}{\sin(\te_0)},
& \te_0 \in S_{2 \pi}\backslash\{0,\pi\}, \\
\f{u'_{-,\te_0}(z,0)}{\cos(\te_0)}, & \te_0 \in S_{2 \pi}\backslash\{\pi/2,3\pi/2\},
 \end{cases}   \lb{6.2} \\
\big(\gamma_{\te_0^{\prime},\te_R^{\prime}}(\Hte - z I)^{-1}\big)_2 f
&= \f{- \sin(\te_R'-\te_R)}{W(u_{+,\te_R}(z,\cdot), u_{-,\te_0}(z,\cdot))}
\big(\ol{u_{-,\te_0}(z,\cdot)}, f\big)_{L^2((0,R); dx)}   \no \\
& \quad \times \begin{cases}  \f{u_{+,\te_R}(z,R)}{\sin(\te_R)},
& \te_R \in S_{2 \pi}\backslash\{0,\pi\}, \\
\f{u'_{+,\te_R}(z,R)}{\cos(\te_R)}, & \te_R \in S_{2 \pi}\backslash\{\pi/2,3\pi/2\},
 \end{cases}    \lb{6.3}
\end{align}
in addition,
\begin{align}
& \gamma_{\te_0,\te_R}(\Hte - z I)^{-1} = 0 \,
\text{ in $\cB\big(L^2((0,R); dx), \bbC^2\big)$},    \lb{6.5a} \\
& \big(\gamma_{\te_0,\te_R}(\Hte - z I)^{-1}\big)_k = 0  \,
\text{ in $\cB\big(L^2((0,R); dx), \bbC\big)$}, \; k=1,2.  \lb{6.5b}
\end{align}
\end{lemma}

Introducing the orthogonal projections in $\bbC^2$,
\begin{equation}
P_1 = \begin{bmatrix} 1 & 0 \\ 0 & 0 \end{bmatrix}, \quad
P_2  = \begin{bmatrix} 0 & 0 \\ 0 & 1 \end{bmatrix},    \lb{6.8}
\end{equation}
one obtains the following Krein-type resolvent formulas (cf.\ \cite[Ch.\ 8]{AG81},
\cite{AB09}, \cite{BGW09}--\cite{CGM10}, \cite{GKMT01}, \cite{GMT98},
\cite{GM08}--\cite{GM10}, \cite{Gr08}, \cite{KO77}--\cite{KS66}, \cite{Ne83},
\cite{Po08}, \cite{PR09}, \cite{Sa65}):

\begin{theorem} \lb{t6.3}
Assume that $\te_0, \te_R, \te_0', \te_R' \in S_{2 \pi}$, let $\Hte$ and
$H_{\theta_0',\theta_R'}$ be defined as in \eqref{2.2a}, and suppose that
$z\in\rho(\Hte) \cap \rho(H_{\theta_0',\theta_R'})$.
Then, with $\Lates (\cdot)$ introduced in \eqref{2.6}, and with $\Ste$ defined
as in \eqref{2.14},
\begin{align}
& (H_{\theta_0',\theta_R'} -z I)^{-1} = (\Hte -z I)^{-1}  \no \\
& \quad - \big[\gamma_{\ol{\theta_0^{\prime}},\ol{\theta_R^{\prime}}}
((\Hte)^* - {\ol z} I)^{-1}\big]^* S_{\te'_0-\te_0,\te'_R-\te_R}^{-1}   \lb{6.21} \\
 & \qquad \times \Big[\Lates (z)\Big]^{-1}
 \big[\gamma_{\theta_0^{\prime},\theta_R^{\prime}}(\Hte - z I)^{-1}\big],
\quad \te_0 \neq \te_0', \; \te_R \neq \te_R'.  \no \\
& (H_{\theta_0,\theta_R'} -z I)^{-1} = (\Hte -z I)^{-1}  \no \\
& \quad - \big[\gamma_{\ol{\theta_0},\ol{\theta_R^{\prime}}}
((\Hte)^* - {\ol z} I)^{-1}\big]^* [\sin(\te'_R-\te_R)]^{-1} P_2   \lb{6.22} \\
 & \qquad \times \Big[\Lambda_{\te_0,\te_R}^{\te_0,\te_R'} (z)\Big]^{-1} P_2
 \big[\gamma_{\theta_0,\theta_R^{\prime}}(\Hte - z I)^{-1}\big],
\quad \te_R \neq \te_R',  \no \\
& (H_{\theta_0',\theta_R} -z I)^{-1} = (\Hte -z I)^{-1}  \no \\
& \quad - \big[\gamma_{\ol{\theta_0'},\ol{\theta_R}}
((\Hte)^* - {\ol z} I)^{-1}\big]^* [\sin(\te_0'-\te_0)]^{-1} P_1  \lb{6.23} \\
 & \qquad \times \Big[\Lambda_{\te_0,\te_R}^{\te_0',\te_R} (z)\Big]^{-1} P_1
 \big[\gamma_{\theta_0',\theta_R}(\Hte - z I)^{-1}\big],
\quad \te_0 \neq \te_0'.  \no
\end{align}
\end{theorem}

\section{Boundary Data Maps, Perturbation Determinants
and Trace Formulas for Schr\"odinger Operators}  \label{s8}

In this section we present our second group of new results, the connection
between boundary data maps, appropriate perturbation determinants, and
trace formulas in the context of self-adjoint one-dimensional Schr\"odinger
operators.

While Theorem \ref{t7.8} appears to be an interesting extension of the classical result, Theorem \ref{t7.1}, it is in general, that is, in the context of
non-self-adjoint operators, not a simple task to verify the hypotheses
\eqref{7.18}--\eqref{7.20} as they involve square root domains. In particular,
it appears to be unknown whether or not $\dom(\Hte)$ and
$\dom((\Hte)^*)$ coincide and hence coincide with $\dom(Q_{\Hte})$, the form
domain of $\Hte$ (assuming $\Hte$ to be non-self-adjoint): This question amounts to solving ``Kato's problem'' in the special case of the non-self-adjoint Schr\"odinger operator $\Hte$ (cf., e.g., and \cite{ADM96}, 
\cite{AMN97}, \cite{Ka62}, \cite{Li62}, \cite{Mc72},
\cite{Mc85}, and \cite{Mc90}), a topic we will return to elsewhere.

To be on safe ground, we now confine ourselves to the special case
of self-adjoint operators $\Hte$ for the remainder of this section: Necessary
and sufficient conditions for $\Hte$ to be self-adjoint are the conditions
\begin{equation}
V\in L^1((0,R); dx) \, \text{ is real-valued},    \lb{7.24}
\end{equation}
and
\begin{equation}
\te_0, \te_R \in [0,2\pi),      \lb{7.25}
\end{equation}
assumed from now on. Then the 2nd representation theorem for densely
defined, semibounded, closed quadratic forms (cf.\ \cite[Sect.\ 6.2.6]{Ka80})
yields that
\begin{align}
\begin{split}
\dom((\Hte - z I_{\cH})^{1/2}) = \dom(|\Hte|^{1/2}) = \dom(Q_{\Hte}),& \\
\quad \te_0, \te_R \in [0,2\pi), \; z \in \bbC\backslash [e_{\te_0,\te_R},\infty),&
\lb{7.26}
\end{split}
\end{align}
where we abbreviated
\begin{equation}
e_{\te_0,\te_R} = \inf (\sigma(\Hte)), \quad \te_0, \te_R \in [0,2\pi).
\end{equation}

Here $(\Hte - z I_{\cH})^{1/2}$ is defined with the help of the spectral theorem and
a choice of a branch cut along $[e_{\te_0,\te_R},\infty)$.
A comparison with \eqref{2.2f}--\eqref{2.2i}, employing the fact that
\begin{align}
& \dom((H_{\theta_0',\theta_R'} - z I_{\cH})^{1/2})
= \dom(|H_{\theta_0',\theta_R'}|^{1/2}) = H^1((0,R)),    \\
& \hspace*{2.58cm} \theta_0',\theta_R' \in [0,2\pi)\backslash\{0, \pi\},
\; z \in \bbC\backslash [e_{\te_0',\te_R'},\infty),  \no \\
& \dom((\Hte - z I_{\cH})^{1/2}) =
\dom(|\Hte|^{1/2}) \subseteq H^1((0,R)),   \\
& \hspace*{3.64cm} \te_0, \te_R \in [0,2\pi),
\; z \in \bbC\backslash [e_{\te_0,\te_R},\infty),   \no
\end{align}
then shows that
\begin{align}
& \ol{(H_{\theta_0',\theta_R'} - z I)^{1/2} (\Hte - z I)^{-1}
(H_{\theta_0',\theta_R'} - z I)^{1/2}}    \no \\
& \quad = \big[(H_{\theta_0',\theta_R'} - z I)^{1/2} (\Hte - z I)^{-1/2}\big]    \no \\
& \qquad \times
\big[(H_{\theta_0',\theta_R'} - {\ol z} I)^{1/2} (\Hte - {\ol z} I)^{-1/2} \big]^*
\in \cB\big(L^2((0,R); dx)\big),    \lb{7.27} \\
& \hspace*{1.7cm} \theta_0',\theta_R' \in [0,2\pi)\backslash\{0, \pi\}, \;
\theta_0,\theta_R \in [0,2\pi), \; z \in \bbC\backslash[e_0,\infty),    \no
\end{align}
where we introduced the abbreviation
\begin{equation}
e_0 =\inf\big(\sigma(\Hte) \cup \sigma( H_{\theta_0',\theta_R'})\big)
= \min (e_{\te_0,\te_R}, e_{\te_0',\te_R'}).
\end{equation}

Moreover, applying Theorem \ref{t6.3} one concludes that actually,
\begin{align}
& \ol{(H_{\theta_0',\theta_R'} - z I)^{1/2} (\Hte - z I)^{-1}
(H_{\theta_0',\theta_R'} - z I)^{1/2}} - I    \no \\
& \quad = - \ol{(H_{\theta_0',\theta_R'} - z I)^{1/2}
\big[(H_{\theta_0',\theta_R'} - z I)^{-1}
- (\Hte - z I)^{-1}\big] (H_{\theta_0',\theta_R'} - z I)^{1/2}}     \no \\
& \qquad \;\, \text{is a finite-rank (and hence a trace class operator)
on $L^2((0,R); dx)$,}    \lb{7.28} \\
& \hspace*{2.95cm}   \theta_0',\theta_R' \in [0,2\pi)\backslash\{0, \pi\}, \;
\theta_0,\theta_R \in [0,2\pi), \; z \in \bbC\backslash[e_0,\infty).   \no
\end{align}

To see the finite-rank property one can argue as follows: By
\eqref{6.1}--\eqref{6.3}, the $\bbC^2$-vector
$\gamma_{\te_0^{\prime},\te_R^{\prime}}(\Hte - z I)^{-1} f$,
$f \in L^2\big((0,R); dx\big)$, is of the type
\begin{equation}
\gamma_{\te_0^{\prime},\te_R^{\prime}}(\Hte - z I)^{-1} f
= \begin{bmatrix}
C_1 \big(\ol{u_{+,\te_R}(z,\cdot)}, f\big)_{L^2((0,R); dx)} \\[2mm]
C_2 \big(\ol{u_{-,\te_0}(z,\cdot)}, f\big)_{L^2((0,R); dx)}
\end{bmatrix},    \lb{7.28a}
\end{equation}
for some $C_j=C_j(z,\te_0',\te_R',\te_0,\te_R)$, $j=1,2$, and hence,
since obviously $u_{+,\te_R}(z,\cdot)$ and $u_{-,\te_0}(z,\cdot)$ belong to
$H^1((0,R))$,
\begin{align}
& \gamma_{\te_0^{\prime},\te_R^{\prime}}(\Hte - z I)^{-1}
(H_{\te_0',\te_R'} -z I)^{1/2} g    \no \\
& \quad = \begin{bmatrix}
C_1 \big(\ol{u_{+,\te_R}(z,\cdot)}, (H_{\te_0',\te_R'} -z I)^{1/2}
g\big)_{L^2((0,R); dx)} \\[2mm]
C_2 \big(\ol{u_{-,\te_0}(z,\cdot)}, (H_{\te_0',\te_R'} -z I)^{1/2}
g\big)_{L^2((0,R); dx)} \end{bmatrix}  \no \\
& \quad = \begin{bmatrix}
C_1 \big([(H_{\te_0',\te_R'} -z I)^{1/2}]^* \ol{u_{+,\te_R}(z,\cdot)},
g\big)_{L^2((0,R); dx)} \\[2mm]
C_2 \big([(H_{\te_0',\te_R'} -z I)^{1/2}]^* \ol{u_{-,\te_0}(z,\cdot)},
g\big)_{L^2((0,R); dx)} \end{bmatrix}, \quad g \in H^1((0,R)),  \lb{7.28b}
\end{align}
extends by continuity to all $g \in L^2\big((0,R); dx\big)$. Similarly, using
\cite[eq.\ (3.54)]{CGM10}, one infers for any $[a_0 \;\, a_R]^\top \in\bbC^2$ that
\begin{align}
\begin{split}
& \big[\gamma_{\te_0',\te_R'} ((\Hte)^*- {\ol z} I)^{-1}\big]^*
[a_0 \;\, a_R]^\top     \\
& \quad = D_1 a_0 u_{+,\te_R}(z,\cdot)
+ D_2 a_R u_{-,\te_0}(z,\cdot) \in H^1((0,R)),   \lb{7.28c}
\end{split}
\end{align}
for some $D_j=D_j(z,\te_0',\te_R',\te_0,\te_R)$, $j=1,2$. Consequently,
\begin{equation}
(H_{\te_0',\te_R'} -z I)^{1/2}
\big[\gamma_{\te_0',\te_R'} ((\Hte)^*- {\ol z} I)^{-1}\big]^*
[a_0 \;\, a_R]^\top \in L^2\big((0,R); dx\big)    \lb{7.28d}
\end{equation}
is well-defined for all $[a_0 \;\, a_R]^\top \in\bbC^2$. Thus, combining \eqref{7.28b}
(for arbitrary $g \in L^2\big((0,R); dx\big)$) and \eqref{7.28d} (for arbitrary
$[a_0 \;\, a_R]^\top \in\bbC^2$) with the finite-rank property of the second terms on
the right-hand sides in \eqref{6.21}--\eqref{6.23} yields the asserted finite-rank
property in \eqref{7.28}.

Thus, the Fredholm determinant, more precisely, the symmetrized
perturbation determinant,
\begin{align}
\begin{split}
{\det}_{L^2((0,R); dx)} \Big(\ol{(H_{\theta_0',\theta_R'} - z I)^{1/2}
(\Hte - z I)^{-1} (H_{\theta_0',\theta_R'} - z I)^{1/2}}\Big),&   \\
\theta_0,\theta_R \in [0,2\pi), \;
\theta_0',\theta_R' \in (0,2\pi)\backslash\{\pi\}, \;
z \in \bbC\backslash[e_0,\infty),&     \lb{7.29}
\end{split}
\end{align}
is well-defined, and an application of Theorem \ref{t7.8} to
$H_{\theta_0',\theta_R'}$ and $\Hte$ yields
\begin{align}
& {\tr}_{L^2((0,R); dx)}\big((H_{\theta_0',\theta_R'} - z I)^{-1}
- (\Hte - z I)^{-1}\big)    \no \\
& \quad =- \f{d}{dz} \ln\Big({\det}_{L^2((0,R); dx)}
\Big(\ol{(H_{\theta_0',\theta_R'} - z I)^{1/2}
(\Hte - z I)^{-1} (H_{\theta_0',\theta_R'} - z I)^{1/2}}\Big)\Big),    \no \\
& \hspace*{3.1cm}   \theta_0,\theta_R \in [0,2\pi), \;
\theta_0',\theta_R' \in (0,2\pi)\backslash\{\pi\}, \;
z \in \bbC\backslash[e_0,\infty),       \lb{7.30}
\end{align}
whenever ${\det}_{L^2((0,R); dx)}\Big(\ol{(H_{\theta_0',\theta_R'} - z I)^{1/2}
(\Hte - z I)^{-1} (H_{\theta_0',\theta_R'} - z I)^{1/2}}\Big) \neq 0$.

Next, we show that the symmetrized (Fredholm) perturbation determinant
\eqref{7.29} can essentially be reduced to the $2 \times 2$ determinant of the
boundary data map $\Lates (z)$:

\begin{theorem} \lb{t8.1}
Assume that $\te_0, \te_R \in [0, 2 \pi)$, $\te_0', \te_R' \in (0, 2 \pi)\backslash\{\pi\}$,
and suppose that $V$ satisfies \eqref{7.24}. Let $\Hte$ and $H_{\theta_0',\theta_R'}$
be defined as in \eqref{2.2a}. Then,
\begin{align}
\begin{split}
& {\det}_{L^2((0,R); dx)}\Big(\ol{(H_{\theta_0',\theta_R'} - z I)^{1/2}
(\Hte - z I)^{-1} (H_{\theta_0',\theta_R'} - z I)^{1/2}}\Big)     \\
& \quad = \f{\sin(\theta_0) \sin(\theta_R)}{\sin(\theta_0') \sin(\theta_R')} \,
{\det}_{\bbC^2} \Big(\Lates (z)\Big),
\quad z \in \bbC\backslash[e_0,\infty).     \lb{7.41}
\end{split}
\end{align}
\end{theorem}
\begin{proof}
Let $z \in \bbC\backslash[e_0,\infty)$.
By \eqref{2.2A} and \eqref{2.2B} it suffices to consider
$\te_0, \te_R \in [0, \pi)$, $\te_0', \te_R' \in (0, \pi)$. Moreover, we will assume
that $\te_0 \neq 0$ and $\te_R \neq 0$. The cases where $\te_0 = 0$ and/or
$\te_R = 0$ follow along the same lines.

In addition, simplifying the proof a bit, we will choose $z<0$, $|z|$ sufficiently large, and
introduce the following convenient abbreviations:
\begin{align}
& H = \Hte, \quad H' = H_{\theta_0',\theta_R'}, \quad
\gamma = \gamma_{\te_0,\te_R}, \quad
\gamma' = \gamma_{\te_0^{\prime},\te_R^{\prime}},    \no \\
& \Lambda (z) = \Lambda_{\te_0',\te_R'}^{\te_0,\te_R} (z),
\quad S = S_{\te_0-\te_0', \te_R-\te_R'} = \begin{pmatrix} \sin(\te_0-\te_0')
& 0 \\ 0 & \sin(\te_R-\te_R') \end{pmatrix},    \no \\
& \check u_+ (z,\cdot) = u_{+,\te_R'} (z,\cdot),
\quad \check u_- (z,\cdot) = u_{-,\te_0'} (z,\cdot),   \lb{7.42} \\
& \check W(z) = W(\check u_+ (z,\cdot), \check u_- (z,\cdot))
= W(u_{+,\te_R'} (z,\cdot), u_{-,\te_0'} (z,\cdot)),    \no \\
& B(z) = (H' - z I)^{1/2} \big[\gamma (H' - z I)^{-1}\big]^* \in
\cB\big(\bbC^2, L^2((0,R); dx)\big).     \no
\end{align}
That $B(z) \in \cB\big(\bbC^2, L^2((0,R); dx)\big)$ follows as in \eqref{7.28c},
\eqref{7.28d}. In addition, we recall that
\begin{align}
& \big[\gamma (H' - z I)^{-1}\big]^* (a_0 \; a_R)^\top
= \f{\sin(\te_0)}{\check W(z)} [\check u'_- (z,0) + \cot(\te_0) \check u_-(z,0)]
a_0 \check u_+ (z,\cdot)    \no \\
& \quad - \f{\sin(\te_R)}{\check W(z)} [\check u'_+ (z,R) - \cot(\te_R) \check u_+(z,R)]
a_R \check u_- (z,\cdot), \quad  (a_0 \; a_R)^\top \in\bbC^2     \lb{7.43}
\end{align}
(cf.\ \cite[eq.\ (3.54)]{CGM10}). Thus,
\begin{align}
B(z) (a_0 \; a_R)^\top
&= \f{\sin(\te_0)}{\check W(z)} [\check u'_- (z,0) + \cot(\te_0) \check u_-(z,0)]
a_0 (H' - z I)^{1/2} \check u_+ (z,\cdot)    \no \\
& \quad - \f{\sin(\te_R)}{\check W(z)} [\check u'_+ (z,R) - \cot(\te_R) \check u_+(z,R)]
a_R (H' - z I)^{1/2} \check u_- (z,\cdot),    \no \\
& \hspace*{6.5cm}  (a_0 \; a_R)^\top \in\bbC^2,     \lb{7.44}
\end{align}
and hence $B(z)^* \in \cB\big(L^2((0,R); dx), \bbC^2\big)$ is given by
\begin{align}
& B(z)^* f \no \\
& \quad = \begin{pmatrix}
\f{\sin(\te_0)}{\check W(z)} [\check u'_- (z,0) + \cot(\te_0) \check u_-(z,0)]
((H' - z I)^{1/2} \check u_+ (z,\cdot), f)_{L^2((0,R); dx)}   \\[2mm]
- \f{\sin(\te_R)}{\check W(z)} [\check u'_+ (z,R) - \cot(\te_R) \check u_+(z,R)]
((H' - z I)^{1/2} \check u_- (z,\cdot), f)_{L^2((0,R); dx)}
\end{pmatrix},     \no \\
& \hspace*{8.2cm}   f \in L^2((0,R); dx).     \lb{7.45}
\end{align}
Using the following version of the Krein-type resolvent formula \eqref{6.21}
\begin{align}
(H - z I)^{-1} = (H' - z I)^{-1} - \big[\gamma (H' - z I)^{-1}\big]^* S^{-1}
\Lates \big[\gamma (H' - z I)^{-1}\big],    \lb{7.46}
\end{align}
one obtains
\begin{align}
& \ol{(H' - z I)^{1/2} (H - z I)^{-1} (H' - z I)^{1/2}}    \no \\
& \quad = I  - \ol{(H' - z I)^{1/2} \big[ (H' - z I)^{-1} - (H - z I)^{-1}\big] (H' - z I)^{1/2}}
\no \\
& \quad = I - B(z) S^{-1} \Lates B(z)^*,     \lb{3.47}
\end{align}
and thus,
\begin{align}
& {\det}_{L^2((0,R); dx)}\Big(\ol{(H' - z I)^{1/2} (H - z I)^{-1} (H' - z I)^{1/2}}\Big)    \no \\
& \quad = {\det}_{L^2((0,R); dx)}\Big(I - B(z) S^{-1} \Lates B(z)^*\Big)   \no \\
& \quad = {\det}_{\bbC^2}\Big(I_2 - S^{-1} \Lates B(z)^* B(z)\Big),     \lb{7.48}
\end{align}
using cyclicity for determinants .

Next we turn to the computation of the $2 \times 2$ matrix $B(z)^* B(z)$: By equations \eqref{7.44}
and \eqref{7.45} one infers
\begin{align}
& B(z)^* B(z) = \big(C_{j,k}(z)\big)_{j,k=1,2},    \lb{7.49} \\
& C_{1,1}(z) = \f{\sin^2(\te_0)}{\check W(z)^2}
[\check u_-' (z,0) + \cot(\te_0) \check u_- (z,0)]^2    \no \\
& \hspace*{1.55cm} \times ((H' -z I)^{1/2} \check u_+ (z,\cdot),
(H' -z I)^{1/2} \check u_+ (z,\cdot))_{L^2((0,R); dx)},    \lb{7.50} \\
& C_{1,2}(z) = - \f{\sin(\te_0) \sin(\te_R)}{\check W(z)^2}    \no \\
& \hspace*{1.75cm} \times [\check u_-' (z,0) + \cot(\te_0) \check u_- (z,0)]
[\check u_+' (z,R) - \cot(\te_R) \check u_+ (z,R)]   \no \\
& \hspace*{1.75cm} \times ((H' -z I)^{1/2} \check u_+ (z,\cdot),
(H' -z I)^{1/2} \check u_- (z,\cdot))_{L^2((0,R); dx)},    \lb{7.51} \\
& C_{2,1}(z) = - \f{\sin(\te_0) \sin(\te_R)}{\check W(z)^2}    \no \\
& \hspace*{1.75cm} \times [\check u_-' (z,0) + \cot(\te_0) \check u_- (z,0)]
[\check u_+' (z,R) - \cot(\te_R) \check u_+ (z,R)]   \no \\
& \hspace*{1.75cm} \times ((H' -z I)^{1/2} \check u_- (z,\cdot),
(H' -z I)^{1/2} \check u_+ (z,\cdot))_{L^2((0,R); dx)},    \lb{7.52} \\
& C_{2,2}(z) = \f{\sin^2(\te_R)}{\check W(z)^2}
[\check u_+' (z,R) - \cot(\te_R) \check u_+ (z,R)]^2    \no \\
& \hspace*{1.55cm} \times ((H' -z I)^{1/2} \check u_- (z,\cdot),
(H' -z I)^{1/2} \check u_- (z,\cdot))_{L^2((0,R); dx)}.     \lb{7.53}
\end{align}

A straightforward, although rather tedious computation then yields the following
simplification of $C_{j,k}(z)$, $j,k=1,2$, and hence of $B(z)^* B(z)$:
\begin{align}
& C_{1,1}(z) = - \f{\sin^2(\te_0 - \te_0')}{\sin^2(\te_0')}
\f{1}{\check u_+'(z,0) + \cot(\te_0')},    \lb{7.54} \\
& C_{1,2}(z) = \f{\sin(\te_0 - \te_0') \sin(\te_R - \te_R')}{\sin(\te_0') \sin(\te_R')}
\f{\check u_-(z,0)}{\check u_-'(z,R) - \cot(\te_R')}     \lb{7.55} \\
& \hspace*{1.1cm} = - \f{\sin(\te_0 - \te_0') \sin(\te_R - \te_R')}{\sin(\te_0') \sin(\te_R')}
\f{\check u_+(z,R)}{\check u_+'(z,0) + \cot(\te_0')}    \lb{7.56} \\
& \hspace*{1.1cm} = C_{2,1}(z),     \lb{7.57} \\
& C_{2,2}(z) = \f{\sin^2(\te_R - \te_R')}{\sin^2(\te_R')}
\f{1}{\check u_-'(z,R) - \cot(\te_R')}.     \lb{7.58}
\end{align}
To arrrive at \eqref{7.54}--\eqref{7.58} one repeatedly uses the identity
\begin{equation}
\cot(x) - \cot(y) = \f{\sin(y-x)}{\sin(y)\sin(x)},    \lb{7.59}
\end{equation}
the following expressions for the Wronskian $\check W$,
\begin{align}
\begin{split}
\check W(z) &= \check u_+(z,R) [\check u_-'(z,R) - \cot(\te_R')]   \\
&= - \check u_-(z,0) [\check u_+'(z,0) + \cot(\te_0')],     \lb{7.60}
\end{split}
\end{align}
and
\begin{align}
& ((H' -z I)^{1/2} \check u_+ (z,\cdot),
(H' -z I)^{1/2} \check u_+ (z,\cdot))_{L^2((0,R); dx)}    \no \\
& \quad = - [\check u_+' (z,0) + \cot(\te_0')],     \lb{7.61} \\
& ((H' -z I)^{1/2} \check u_+ (z,\cdot),
(H' -z I)^{1/2} \check u_- (z,\cdot))_{L^2((0,R); dx)}    \no \\
& \quad = - \check u_-(z,0) [\check u_+' (z,0) + \cot(\te_0')]
= \check u_+(z,R) [\check u_-' (z,R) - \cot(\te_R')],     \lb{7.62} \\
& \quad = ((H' -z I)^{1/2} \check u_- (z,\cdot),
(H' -z I)^{1/2} \check u_+ (z,\cdot))_{L^2((0,R); dx)},    \no \\
& ((H' -z I)^{1/2} \check u_- (z,\cdot),
(H' -z I)^{1/2} \check u_- (z,\cdot))_{L^2((0,R); dx)}    \no \\
& \quad = [\check u_-' (z,R) - \cot(\te_R')].    \lb{7.63}
\end{align}
Relations \eqref{7.61}--\eqref{7.63} are a consequence of one integration by
parts in \eqref{2.2f}.

Finally, we compute $\Lambda (z) S$, starting with \eqref{2.20a} and \eqref{7.42}:
\begin{align}
\Lambda (z) S &= \big(K_{j,k}(z)\big)_{j,k=1,2},    \lb{7.64} \\
K_{1,1}(z) &= \f{\sin(\te_0-\te_0') \sin(\te_0)}{\sin(\te_0')}
\f{\check u_+'(z,0) + \cot(\te_0)}{\check u_+'(z,0) + \cot(\te_0')}   \no \\
&= C_{1,1}(z) + \f{\sin(\te_0-\te_0') \sin(\te_0)}{\sin(\te_0')},    \lb{7.65} \\
K_{1,2}(z) &= - \f{\sin(\te_R-\te_R') \sin(\te_0)}{\sin(\te_R')}
\f{\check u_-(z,0) + \cot(\te_0) \check u_-(z,0)}{\check u_-'(z,R) - \cot(\te_R')}  \no \\
& = \f{\sin(\te_0-\te_0') \sin(\te_R-\te_R')}{\sin(\te_0') \sin(\te_R')}
\f{\check u_-(z,0)}{\check u_-'(z,R) - \cot(\te_R')}   \no \\
& =C_{1,2}(z) = C_{2,1}(z)   \lb{7.66} \\
& = - \f{\sin(\te_0 - \te_0') \sin(\te_R - \te_R')}{\sin(\te_0') \sin(\te_R')}
\f{\check u_+(z,R)}{\check u_+'(z,0) + \cot(\te_0')}   \no \\
& = K_{2,1}(z),     \lb{7.67} \\
K_{2,2}(z) &= \f{\sin(\te_R-\te_R') \sin(\te_R)}{\sin(\te_R')}
\f{\check u_-'(z,R) - \cot(\te_R)}{\check u_-'(z,R) + \cot(\te_R')}   \no \\
&= C_{2,2}(z) + \f{\sin(\te_R-\te_R') \sin(\te_R)}{\sin(\te_R')}.    \lb{7.68}
\end{align}
In particular,
\begin{equation}
\Lambda (z) S = B(z)^* B(z) + \begin{pmatrix}
\f{\sin(\te_0-\te_0') \sin(\te_0)}{\sin(\te_0')} & 0 \\ 0
& \f{\sin(\te_R-\te_R') \sin(\te_R)}{\sin(\te_R')} \end{pmatrix}.    \lb{7.69}
\end{equation}
An insertion of \eqref{7.69} into \eqref{7.48} finally yields
\begin{align}
& {\det}_{L^2((0,R); dx)}\Big(\ol{(H' - z I)^{1/2} (H - z I)^{-1} (H' - z I)^{1/2}}\Big)
\no \\
& \quad = {\det}_{\bbC^2}\Big(I_2 - S^{-1} \Lates B(z)^* B(z)\Big),    \no \\
& \quad = {\det}_{\bbC^2}\left(I_2 - [\Lambda(z) S]^{-1}\left[\Lambda (z) S
- \begin{pmatrix}
\f{\sin(\te_0-\te_0') \sin(\te_0)}{\sin(\te_0')} & 0 \\ 0
& \f{\sin(\te_R-\te_R') \sin(\te_R)}{\sin(\te_R')} \end{pmatrix}\right]\right)   \no \\
& \quad = {\det}_{\bbC^2}\left([\Lambda(z) S]^{-1} \begin{pmatrix}
\f{\sin(\te_0-\te_0') \sin(\te_0)}{\sin(\te_0')} & 0 \\ 0
& \f{\sin(\te_R-\te_R') \sin(\te_R)}{\sin(\te_R')} \end{pmatrix}\right)   \no \\
& \quad = {\det}_{\bbC^2}\Big(\Lates (z)\Big)
 {\det}_{\bbC^2}\big(S^{-1}\big)
 \f{\sin(\te_0-\te_0') \sin(\te_R-\te_R') \sin(\te_0) \sin(\te_R)}{\sin(\te_0') \sin(\te_R')}
 \no \\
& \quad =  \f{\sin(\te_0) \sin(\te_R)}{\sin(\te_0') \sin(\te_R')} \,
{\det}_{\bbC^2}\Big(\Lates (z)\Big).    \lb{7.70}
\end{align}
\end{proof}

Since the Fredholm determinant on the left-hand side of \eqref{7.41} vanishes for
$\te_0 = 0$ and/or $\te_R = 0$, we now briefly consider the nullspace of the operators involved:.

\begin{lemma} \lb{l8.2}
Assume that $\te_0, \te_R \in [0, 2 \pi)$, $\te_0', \te_R' \in (0, 2 \pi)\backslash\{\pi\}$,
$z \in \bbC\backslash[e_0,\infty)$,
and suppose that $V$ satisfies \eqref{7.24}. Let $\Hte$ and $H_{\theta_0',\theta_R'}$
be defined as in \eqref{2.2a}. Then recalling the factorization,
\begin{align}
& \ol{(H_{\theta_0',\theta_R'} - z I)^{1/2}
(\Hte - z I)^{-1} (H_{\theta_0',\theta_R'} - z I)^{1/2}}    \lb{7.71} \\
& \quad =
(H_{\theta_0',\theta_R'} - z I)^{1/2}
(\Hte - z I)^{-1/2} \big[(H_{\theta_0',\theta_R'} - \ol z I)^{1/2}
(\Hte - \ol z I)^{-1/2}\big]^*    \no
\end{align}
one obtains
\begin{align}
& \ker\big(\big[(H_{\theta_0',\theta_R'} - \ol z I)^{1/2}
(\Hte - \ol z I)^{-1/2}\big]^*\big)   \no \\
& \quad = \big\{f \in L^2((0,R); dx) \, \big| \,
f = (H_{\theta_0',\theta_R'} - z I)^{1/2} \psi (z,\cdot);    \no \\
& \qquad \quad  \psi (z,\cdot), \psi'(z,\cdot) \in AC([0,R]); \,
-\psi''(z,\cdot) + (V(\cdot) - z) \psi(z,\cdot) = 0; \no \\
& \qquad \quad
\psi'(z,R) - \cot(\te_R') \psi(z,R) =0 \text{ if } \te_0 = 0, \, \te_R \neq 0; \lb{7.72} \\
& \qquad \quad
\psi'(z,0) + \cot(\te_0') \psi(z,0) =0 \text{ if } \te_0 \neq 0, \, \te_R = 0   \no \\
& \qquad \quad
\text{no boundary conditions on $\psi(z,\cdot)$ if $\te_0 = \te_R = 0$}\big\}, \no
\end{align}
in particular,
\begin{align}
\begin{split}
& \dim\big(\ker\big(\big[(H_{\theta_0',\theta_R'} - \ol z I)^{1/2}
(\Hte - \ol z I)^{-1/2}\big]^*\big)\big)     \\
& \quad = \begin{cases}
2, & \te_0 = \te_R =0, \\
1, & \te_0 =0, \, \te_R \neq 0 \, \text{ or } \,  \te_0 \neq 0, \, \te_R = 0.
\end{cases}
\end{split}    \lb{7.73}
\end{align}
\end{lemma}
\begin{proof}
Let $z \in \bbC\backslash[e_0,\infty)$.
To determine the precise characterization of the nullspace in \eqref{7.72} one can argue as follows: Suppose first that $\te_0 = \te_R = 0$ and that
\begin{align}
\begin{split}
& f \bot \, \ran\big((H_{\theta_0',\theta_R'} - \ol z I)^{1/2}(H_{0,0} - \ol z I)^{-1/2}\big),  \\
& g \in \ran\big((H_{\theta_0',\theta_R'} - \ol z I)^{1/2}(H_{0,0} - \ol z I)^{-1/2}\big),
\lb{7.74}
\end{split}
\end{align}
implying
\begin{equation}
g = (H_{\theta_0',\theta_R'} - \ol z I)^{1/2} h \, \text{ for some } \,
h \in H^1_0((0,R)),     \lb{7.75}
\end{equation}
and hence $h(0)=h(R)=0$.
Thus, introducing $\psi(z,\cdot) = (H_{\theta_0',\theta_R'} - z I)^{-1/2} f$, one obtains
using \eqref{2.2f} again,
\begin{align}
0 &= (g,f)_{L^2((0,R); dx)}   \no \\
&= \big((H_{\theta_0',\theta_R'} - \ol z I)^{1/2} h, (H_{\theta_0',\theta_R'} - z I)^{1/2}
(H_{\theta_0',\theta_R'} - z I)^{-1/2} f\big)_{L^2((0,R); dx)}   \no \\
&= \big((H_{\theta_0',\theta_R'} - \ol z I)^{1/2} h, (H_{\theta_0',\theta_R'} - z I)^{1/2}
\psi(z)\big)_{L^2((0,R); dx)}   \no \\
&= \int_0^R dx \, [\ol{h'(x)} \psi'(z,x) + (V(x)-z) \ol{h(x)} \psi(z,x)]   \no \\
& \quad - \cot(\te_0') \ol{h(0)} \psi(z,0) - \cot(\te_R') \ol{h(R)} \psi(z,R)   \no \\
&= \ol{h(x)} \psi'(z,x)\big|_0^R
+ \int_0^R dx \, \ol{h(x)} [\psi''(z,x) + (V(x)-z) \psi(z,x)]   \no \\
&= \int_0^R dx \, \ol{h(x)} [\psi''(z,x) + (V(x)-z) \psi(z,x)], \quad h \in H^1_0((0,R)).
\lb{7.76}
\end{align}
Hence one concludes that
\begin{equation}
\psi(z,\cdot), \, \psi'(z,\cdot) \in AC([0,R]),    \lb{7.77}
\end{equation}
and that
\begin{equation}
\psi''(z,\cdot) + (V(\cdot)-z) \psi(z,\cdot) = 0 \, \text{ in the sense of distributions.}
\lb{7.78}
\end{equation}
As $g \in \ran\big((H_{\theta_0',\theta_R'} - \ol z I)^{1/2}(H_{0,0} - \ol z I)^{-1/2}\big)$
was chosen arbitrarily, one concludes that any element $f$ in
$\ker\big(\big[(H_{\theta_0',\theta_R'} - \ol z I)^{1/2}
(H_{0,0} - \ol z I)^{-1/2}\big]^*\big)$ is of the form
\begin{equation}
f = (H_{\theta_0',\theta_R'} - z I)^{1/2} \psi(z,\cdot).    \lb{7.79}
\end{equation}
The fact that $\psi(z,\cdot)$ satisfies no boundary conditions then shows that the dimension of the nullspace in \eqref{7.72} is precisely two if $\te_0=\te_R=0$.

Next we consider the case $\te_0=0$, $\te_R \neq 0$ (the case
$\te_0 \neq 0$, $\te_R = 0$ being completely analogous): Again we assume
\begin{align}
\begin{split}
& f \bot \, \ran\big((H_{\theta_0',\theta_R'} - \ol z I)^{1/2}(H_{0,\te_R} - \ol z I)^{-1/2}\big),  \\
& g \in \ran\big((H_{\theta_0',\theta_R'} - \ol z I)^{1/2}(H_{0,\te_R} - \ol z I)^{-1/2}\big),
\lb{7.80}
\end{split}
\end{align}
implying
\begin{equation}
g = (H_{\theta_0',\theta_R'} - \ol z I)^{1/2} h \, \text{ for some } \,
h \in H^1 ((0,R)) \, \text{ with $h(0)=0$.}     \lb{7.81}
\end{equation}
Introducing once more $\psi(z,\cdot) = (H_{\theta_0',\theta_R'} - z I)^{-1/2} f$, one obtains again via \eqref{2.2f} that
\begin{align}
0 &= (g,f)_{L^2((0,R); dx)}   \no \\
&= \big((H_{\theta_0',\theta_R'} - \ol z I)^{1/2} h, (H_{\theta_0',\theta_R'} - z I)^{1/2}
(H_{\theta_0',\theta_R'} - z I)^{-1/2} f\big)_{L^2((0,R); dx)}   \no \\
&= \big((H_{\theta_0',\theta_R'} - \ol z I)^{1/2} h, (H_{\theta_0',\theta_R'} - z I)^{1/2}
\psi(z)\big)_{L^2((0,R); dx)}   \no \\
&= \int_0^R dx \, [\ol{h'(x)} \psi'(z,x) + (V(x)-z) \ol{h(x)} \psi(z,x)]   \no \\
& \quad - \cot(\te_0') \ol{h(0)} \psi(z,0) - \cot(\te_R') \ol{h(R)} \psi(z,R)   \no \\
&= \ol{h(x)} \psi'(z,x)\big|_0^R
+ \int_0^R dx \, \ol{h(x)} [\psi''(z,x) + (V(x)-z) \psi(z,x)]   \no \\
& \quad - \cot(\te_R') \ol{h(R)} \psi(z,R)   \no \\
&= \ol{h(R)} [\psi'(z,R) - \cot(\te_R') \psi(z,R)]    \lb{7.82} \\
& \quad + \int_0^R dx \, \ol{h(x)} [\psi''(z,x) + (V(x)-z) \psi(z,x)],
\quad h \in H^1((0,R)), \; h(0)=0.   \no
\end{align}
Choosing temporarily $h \in H^1_0((0,R))$, one obtains
\begin{equation}
\int_0^R dx \, \ol{h(x)} [\psi''(z,x) + (V(x)-z) \psi(z,x)],
\quad h \in H^1_0((0,R)),    \lb{7.83}
\end{equation}
and hence again concludes that
\begin{equation}
\psi(z,\cdot), \, \psi'(z,\cdot) \in AC([0,R]),    \lb{7.84}
\end{equation}
and that
\begin{equation}
\psi''(z,\cdot) + (V(\cdot)-z) \psi(z,\cdot) = 0 \, \text{ in the sense of distributions.}
\lb{7.85}
\end{equation}
Taking into account \eqref{7.85} in \eqref{7.82} then yields
\begin{equation}
0= \ol{h(R)} [\psi'(z,R) - \cot(\te_R') \psi(z,R)],
\quad h \in H^1((0,R)), \; h(0)=0,     \lb{7.86}
\end{equation}
implying
\begin{equation}
[\psi'(z,R) - \cot(\te_R') \psi(z,R)] = 0.   \lb{7.87}
\end{equation}
As before, this proves \eqref{7.72} in the case $\te_0=0$, $\te_R \neq 0$. The
boundary condition \eqref{7.87} then yields that the nullspace \eqref{7.72} is
one-dimensional in this case.
\end{proof}

\begin{remark} \lb{r8.3}
We emphasize the interesting fact that relation \eqref{7.41} represents yet another
reduction of an infinite-dimensional Fredholm determinant (more precisely, a
symmetrized perturbation determinant) to a finite-dimensional determinant. This is
analogous to the following well-known situations: \\
$(i)$ The Jost--Pais formula \cite{JP51} in the context of half-line Schr\"odinger
operators (relating the perturbation determinant of the
corresponding Birman--Schwinger kernel with the Jost function and hence a Wronski
determinant).  \\
$(ii)$ Schr\"odinger operators on the real line \cite{Ne80} (relating the perturbation determinant of the corresponding Birman--Schwinger kernel with the transmission coefficient and hence again a Wronski determinant). \\
$(iii)$ One-dimensional periodic Schr\"odinger operators \cite{Ma55} (relating the Floquet discriminant with an appropriate Fredholm determinant). \\
These cases, and much more general situations in connection with semi-separable
integral kernels (which typically apply to one-dimensional differential and difference
operators with matrix-valued coefficients) were studied in great deal in \cite{GM03}
(see also \cite{GW95} and the multi-dimensional discussion in \cite{GMZ07}).
\end{remark}

We conclude this section by pointing out that determinants (especially, 
$\zeta$-function regularized determinants) for various elliptic boundary value problems on compact intervals (including cases with regular singular coefficients) have received considerable attention and we refer, for 
instance, to Burghelea, Friedlander, and Kappeler \cite{BFK95}, 
Dreyfus and Dym \cite{DD78}, Forman \cite{Fo92}, 
Kirsten, Loya, and Park \cite{KLP08}, Kirsten and McKane \cite{KM04}, 
Lesch \cite{Le98}, Lesch and Tolksdorf \cite{LT98}, 
Lesch and Vertman \cite{LV10}, and Levit and Smilansky \cite{LS77} in this 
context.

\section{Trace Formulas and the Spectral Shift Function}  \label{s4}

In this section we derive the trace formula for the resolvent difference of
$\Hte$ and $H_{\theta_0',\theta_R'}$ in terms of the spectral shift function
$\xi(\cdot; H_{\theta_0',\theta_R'}, \Hte)$ and establish the connection between
$\Lates (\cdot)$ and $\xi(\cdot; H_{\theta_0',\theta_R'}, \Hte)$.

To prepare the ground for the basic trace formula we now state the following
fact (which does not require $\Hte$ and $H_{\theta_0',\theta_R'}$ to be self-adjoint):

\begin{lemma} \lb{l4.1}
Assume that $\te_0, \te_R, \te_0', \te_R' \in S_{2 \pi}$, and let $\Hte$ and
$H_{\theta_0',\theta_R'}$ be defined as in \eqref{2.2a}. Then, with
$\Lates (\cdot) S_{\te_0'-\te_0, \te_R'-\te_R}$ given by \eqref{3.33aa},
\begin{align}
& \gamma_{\te_0', \te_R'} (\Hte - z I)^{-1} \big[\gamma_{\ol{\te_0'},\ol{\te_R'}}
(\Hte^* - {\ol z} I)^{-1}\big]^*
= \f{d}{dz} \Big(\Lates (z) S_{\te_0'-\te_0, \te_R'-\te_R}\Big),  \no \\
& \hspace*{9cm}   z\in\rho(\Hte).    \lb{4.1}
\end{align}
\end{lemma}
\begin{proof}
Employing the resolvent equation for $\Hte^*$, one verifies that
\begin{align}
& \f{d}{dz} \gamma_{\te_0', \te_R'} \big[\gamma_{\ol{\te_0'},\ol{\te_R'}}
(\Hte^* - {\ol z} I)^{-1}\big]^* = \gamma_{\te_0', \te_R'}
\big[\gamma_{\ol{\te_0'},\ol{\te_R'}} (\Hte^* - {\ol z} I)^{-2}\big]^*   \no \\
& \quad = \gamma_{\te_0', \te_R'} (\Hte - z I)^{-1}
\big[\gamma_{\ol{\te_0'},\ol{\te_R'}} (\Hte^* - {\ol z} I)^{-1}\big]^*.   \lb{4.2}
\end{align}
Together with \eqref{3.33aa} this proves \eqref{4.1}.
\end{proof}

Combining Theorems \ref{t6.3} and \ref{t7.8} with Lemma \ref{l4.1} then yields the following result:

\begin{theorem} \lb{t4.2}
Assume that $\te_0, \te_R, \te_0', \te_R' \in [0, 2 \pi)$,
and suppose that $V$ satisfies \eqref{7.24}. Let $\Hte$ and $H_{\theta_0',\theta_R'}$
be defined as in \eqref{2.2a}. Then,
\begin{align}
&  \tr_{L^2((0,R); dx)}\big((H_{\theta_0',\theta_R'} -z I)^{-1}
- (\Hte -z I)^{-1}\big)    \no \\
& \quad = - \tr_{\bbC^2}\bigg(\Big[\Lates (z)\Big]^{-1}
\f{d}{dz} \Big[\Lates (z)\Big]\bigg)    \no \\
& \quad = - \f{d}{dz} \ln\Big({\det}_{\bbC^2}\Big(\Lates (z)\Big)\Big), \quad
z \in \bbC\backslash[e_0,\infty).   \lb{4.3}
\end{align}
If, in addition, $\te_0', \te_R' \in (0, 2 \pi)\backslash\{\pi\}$, then
\begin{align}
&  \tr_{L^2((0,R); dx)}\big((H_{\theta_0',\theta_R'} -z I)^{-1}
- (\Hte -z I)^{-1}\big)   \no \\
& = - \f{d}{dz} \ln\Big({\det}_{L^2((0,R); dx)}\Big(\ol{(H_{\theta_0',\theta_R'} - z I)^{1/2}
(\Hte - z I)^{-1} (H_{\theta_0',\theta_R'} - z I)^{1/2}}\Big)\Big),    \no \\
& \hspace*{9cm}  z \in \bbC\backslash[e_0,\infty).     \lb{4.4}
\end{align}
\end{theorem}
\begin{proof}
The second equality in \eqref{4.3} is obvious. Next, we temporarily suppose
that $\te_0 \neq \te_0'$ and
$\te_R \neq \te_R'$. Then  the first equality in \eqref{4.3} follows upon taking
the trace in \eqref{6.21}, using cyclicity of the trace, and applying \eqref{4.1}
(keeping in mind that $S_{\te_0'-\te_0, \te_R'-\te_R}$ is invertible and
$z$-independent). The remaining cases where
$\te_0 = \te_0'$ or $\te_R = \te_R'$ follow similarly (the case where
$\te_0 = \te_0'$ and $\te_R = \te_R'$ instantly follows from \eqref{2.7a}).

Relation \eqref{4.4} follows from \eqref{7.21}, \eqref{7.41}, and \eqref{4.3}.
\end{proof}

Next, we note that the rank-two behavior of the difference of the
resolvents of $\Hte$ and $H_{\theta_0',\theta_R'}$ displayed in Theorem \ref{t6.3}
permits one to define the spectral shift function
$\xi(\, \cdot \,; H_{\theta_0',\theta_R'}, \Hte)$ associated with the pair
$(H_{\theta_0',\theta_R'}, \Hte)$. Using the standard normalization in the context of
self-adjoint operators bounded from below,
\begin{equation}
\xi(\, \cdot \,; H_{\theta_0',\theta_R'}, \Hte) = 0, \quad
\lambda < e_0 =\inf\big(\sigma(\Hte) \cup \sigma( H_{\theta_0',\theta_R'})\big),
\lb{4.5}
\end{equation}
Krein's trace formula (see, e.g., \cite[Ch.\ 8]{Ya92}, \cite{Ya07}) reads
\begin{align}
\begin{split}
& \tr_{L^2((0,R); dx)}\big((H_{\theta_0',\theta_R'} -z I)^{-1}
- (\Hte -z I)^{-1}\big)   \\
& \quad = - \int_{[e_0,\infty)}
\f{\xi(\lambda; H_{\theta_0',\theta_R'}, \Hte) \, d\lambda}{(\lambda - z)^2},
\quad z\in\rho(\Hte)\cap \rho(H_{\theta_0',\theta_R'})\big),    \lb{4.6}
\end{split}
\end{align}
where $\xi(\cdot \, ; H_{\theta_0',\theta_R'}, \Hte)$ satisfies
\begin{equation}
\xi(\cdot \, ; H_{\theta_0',\theta_R'}, \Hte)
\in L^1\big(\bbR; (\lambda^2 + 1)^{-1} d\lambda\big).
\end{equation}
Since the spectra of $\Hte$ and $H_{\theta_0',\theta_R'}$ are purely discrete,
$\xi(\, \cdot \,; H_{\theta_0',\theta_R'}, \Hte)$ is a piecewise constant function
on $\bbR$ with jumps at the eigenvalues of $\Hte$ and $H_{\theta_0',\theta_R'}$,
the jumps corresponding to the multiplicity of the eigenvalue in question.

Moreover, $\xi(\cdot \, ; H_{\theta_0',\theta_R'}, \Hte)$ permits a representation
in terms of nontangential boundary
values to the real axis of ${\det}_{\bbC^2}\Big(\Lates (\cdot)\Big)$
(resp., of the symmetrized perturbation determinant \eqref{7.29}), to be
described in Theorem \ref{t4.3}.

Since by \eqref{2.2A} and \eqref{2.2B} it suffices to consider
$\te_0, \te_R \in [0,\pi)$ when considering the operator $\Hte$, we will restrict
the boundary condition parameters accordingly next:

\begin{theorem} \lb{t4.3}
Assume that $\te_0, \te_R \in [0, \pi)$, $\te_0', \te_R' \in (0, \pi)$,
and suppose that $V$ satisfies \eqref{7.24}. Let $\Hte$ and $H_{\theta_0',\theta_R'}$
be defined as in \eqref{2.2a}. Then,
\begin{align}
\begin{split}
\xi(\lambda; H_{\theta_0', \theta_R'}, \Hte) =
\pi^{-1} \lim_{\varepsilon \downarrow 0} \Im\Big(
\ln\Big(\eta(\te_0,\te_R) \,
{\det}_{\bbC^2}\Big(\Lates (\lambda + i \varepsilon)\Big)\Big)\Big)&    \lb{4.7} \\
\text{ for a.e.\ $\lambda \in\bbR$,}&
\end{split}
\end{align}
where
\begin{equation}
\eta(\te_0,\te_R) = \begin{cases}
\;\;\, 1, & \te_0, \te_R \in (0,\pi), \; \te_0=\te_R=0, \\
-1, & \te_0=0, \, \te_R \in (0,\pi), \; \te_0 \in (0,\pi), \, \te_R=0.
\end{cases}         \lb{4.8}
\end{equation}
\end{theorem}
\begin{proof}
We recall  the definition of
$e_0 =\inf\big(\sigma(\Hte) \cup \sigma( H_{\theta_0',\theta_R'})\big)$ in \eqref{4.5}.

Combining \eqref{4.3} and \eqref{4.6} one obtains
\begin{align}
\begin{split}
\f{d}{dz} \ln\Big(\eta(\te_0,\te_R) \, {\det}_{\bbC^2}\Big(\Lates (z)\Big)\Big)
= \int_{[e_0,\infty)}
\f{\xi(\lambda; H_{\theta_0',\theta_R'}, \Hte) \, d\lambda}{(\lambda - z)^2},&  \\
z\in\rho(\Hte)\cap \rho(H_{\theta_0',\theta_R'})\big),&      \lb{4.9}
\end{split}
\end{align}
since $\eta(\te_0,\te_R)$ is $z$-independent.

Next, combining \eqref{2.2d} and \eqref{2.43}, and using the fact that $\phi(z,x)$
and $\theta(z,x)$ are both real-valued for $z, x \in \bbR$, one concludes that
$\Delta(z,R,\te_0,\te_R)$, and hence ${\det}_{\bbC^2}\Big(\Lates (z)\Big)$ are
real-valued for $z\in\bbR$ and $\te_0, \te_R, \te_0', \te_R' \in [0,\pi)$. Moreover,
using the fact that
\begin{equation}
{\det}_{\bbC^2}\Big(\Lates (z)\Big) \neq 0, \quad z < e_0,    \lb{4.10}
\end{equation}
and invoking the asymptotic behavior \eqref{2.46} as $z\downarrow 0$, one actually concludes that
\begin{equation}
{\det}_{\bbC^2}\Big(\eta(\te_0,\te_R) \, \Lates (z)\Big) > 0, \quad z < e_0.  \lb{4.11}
\end{equation}

Integrating \eqref{4.9} with respect to the $z$-variable along the real axis from
$z_0$ to $z$, assuming $z<z_0<e_0$, one obtains
\begin{align}
& \ln\Big(\eta(\te_0,\te_R) \, {\det}_{\bbC^2}\Big(\Lates (z)\Big)\Big) -
\ln\Big(\eta(\te_0,\te_R) \, {\det}_{\bbC^2}\Big(\Lates (z_0)\Big)\Big)   \no \\
& \quad = \int_{z_0}^z d\zeta
\int_{[e_0,\infty)} \f{\xi(\lambda; H_{\theta_0',\theta_R'}, \Hte) \, d\lambda}
{(\lambda - \zeta)^2}  \no \\
& \quad = \int_{z_0}^z d\zeta
\int_{[e_0,\infty)} \f{[\xi_+ (\lambda; H_{\theta_0',\theta_R'}, \Hte)
- \xi_- (\lambda; H_{\theta_0',\theta_R'}, \Hte)] \, d\lambda}
{(\lambda - \zeta)^2}   \no \\
& \quad = \int_{[e_0,\infty)} [\xi_+ (\lambda; H_{\theta_0',\theta_R'}, \Hte)
- \xi_- (\lambda; H_{\theta_0',\theta_R'}, \Hte)] \, d\lambda
\int_{z_0}^z \f{d\zeta}{(\lambda - \zeta)^2}   \no \\
& \quad = \int_{[e_0,\infty)} \xi(\lambda; H_{\theta_0',\theta_R'}, \Hte) \, d\lambda
\bigg(\f{1}{\lambda - z} - \f{1}{\lambda - z_0}\bigg), \quad z<z_0<e_0.   \lb{4.12}
\end{align}
Here we split $\xi$ into its positive and negative parts, $\xi_{\pm} = [|\xi| \pm \xi]/2$,
and applied the Fubini--Tonelli theorem to interchange the integrations with respect
to $\lambda$ and $\zeta$. Moreover, we chose the branch of $\ln(\cdot)$ such
that $\ln(x) \in \bbR$ for $x>0$, compatible with the normalization of
$\xi(\, \cdot \,; H_{\theta_0',\theta_R'}, \Hte)$ in \eqref{4.5}.

An analytic continuation of the first and last line of \eqref{4.12} with respect to $z$ then
yields
\begin{align}
& \ln\Big(\eta(\te_0,\te_R) \, {\det}_{\bbC^2}\Big(\Lates (z)\Big)\Big) -
\ln\Big(\eta(\te_0,\te_R) \, {\det}_{\bbC^2}\Big(\Lates (z_0)\Big)\Big)   \no \\
& \quad = \int_{[e_0,\infty)} \xi(\lambda; H_{\theta_0',\theta_R'}, \Hte) \, d\lambda
\bigg(\f{1}{\lambda - z} - \f{1}{\lambda - z_0}\bigg),
\quad z \in \bbC \backslash [e_0,\infty).   \lb{4.13}
\end{align}

 Since by \eqref{4.11},
\begin{equation}
\ln\Big(\eta(\te_0,\te_R) \, {\det}_{\bbC^2}\Big(\Lates (z_0)\Big)\Big) \in \bbR,
\quad z_0 < e_0,     \lb{4.14}
\end{equation}
the Stieltjes inversion formula separately applied to the absolutely continuous
measures $\xi_{\pm} (\lambda; H_{\theta_0',\theta_R'}, \Hte) \, d\lambda$
(cf., e.g., \cite[p.\ 328]{AD56}, \cite[App.\ B]{We80}), then yields \eqref{4.7}.
\end{proof}

\medskip

\noindent {\bf Acknowledgments.} 
We are indebted to Steve Clark, Steve Hofmann, Alan McIntosh, and 
Marius Mitrea for helpful discussions. 
   
Fritz Gesztesy gratefully acknowledges the
kind invitation and hospitality of the Department of Mathematics of
the Western Michigan University, Kalamazoo, during a week in April of 2010, where the early parts of this paper were developed.


\end{document}